\documentclass[a4paper,sort&compress,3p]{elsarticle}

\usepackage{mathtools}
\usepackage{amsthm}
\usepackage{amsmath}
\usepackage{amssymb}

\newtheorem{ass}{Assumption}
\newtheorem{thm}{Theorem}
\newtheorem{rmk}{Remark}
\newtheorem{problem}{Problem}
\newtheorem{lem}{Lemma}

\newcommand{\f}[1]{\mathbf{#1}}
\newcommand{\norm}[1]{\left\lVert#1\right\rVert}

\usepackage{setspace}
\usepackage{multicol}

\usepackage{placeins}

\usepackage{graphicx}
\usepackage{subcaption}

\usepackage{xcolor} 

\usepackage{accents} 
\usepackage{mathabx}

\usepackage{stmaryrd}

\newcommand{\jump}[1]{\left\llbracket#1\right\rrbracket}
\newcommand{\average}[1]{\left\{#1\right\}}

\journal{}

\begin{document}
\begin{frontmatter}

\title{An approximate $C^1$ multi-patch space for isogeometric analysis with a comparison to Nitsche's method} 

\author[JKU]{Pascal Weinm\"uller}
\ead{pascal.weinmueller@jku.at}
\author[JKU,RICAM]{Thomas Takacs}
\ead{thomas.takacs@ricam.oeaw.ac.at}
\address[JKU]{Institute of Applied Geometry, Johannes Kepler University Linz, Altenberger Str. 69, 4040 Linz, Austria}
\address[RICAM]{Johann Radon Institute for Computational and Applied Mathematics, Austrian Academy of Sciences, Altenberger Str. 69, 4040 Linz, Austria}

\begin{abstract}
We present an approximately $C^1$-smooth multi-patch spline construction which can be used in isogeometric analysis (IGA). The construction extends the one presented in~\cite{WEINMULLER2021114017} for two-patch domains. 
A key property of IGA is that it is simple to achieve high order smoothness within a single patch. However, to represent more complex geometries one often uses a multi-patch construction. In this case, the global continuity for the basis functions is in general only $C^0$. Therefore, to obtain $C^1$-smooth isogeometric functions, a special construction for the basis is needed. Such spaces are of interest when solving numerically fourth-order problems, such as the biharmonic equation or Kirchhoff-Love plate/shell formulations, using an isogeometric Galerkin method.

Isogeometric spaces that are globally $C^1$ over multi-patch domains can be constructed as in~\cite{collin2016analysis, kapl2017dimension, kapl2018construction, kapl2019isogeometric, kapl2019argyris}. The constructions require geometry parametrizations that satisfy certain constraints along the interfaces, so-called analysis-suitable $G^1$ parametrizations. To allow $C^1$ spaces over more general multi-patch parametrizations, one needs to increase the polynomial degree and/or to relax the $C^1$ conditions. Thus, we define function spaces that are not exactly $C^1$ but only approximately. We adopt the construction for two-patch domains, as developed in~\cite{WEINMULLER2021114017}, and extend it to more general multi-patch domains.

We employ the construction for a biharmonic model problem and compare the results with Nitsche's method. We compare both methods over complex multi-patch domains with non-trivial interfaces. The numerical tests indicate that the proposed construction converges optimally under $h$-refinement, comparable to the solution using Nitsche's method. In contrast to weakly imposing coupling conditions, the approximate $C^1$ construction is explicit and no additional terms need to be introduced to stabilize the method/penalize the jump of the derivative at the interface. Thus, the new proposed method can be used more easily as no parameters need to be estimated.
\end{abstract}

\begin{keyword}
fourth order partial differential equation \sep biharmonic equation \sep geometric continuity \sep $C^1$ continuity \sep approximate $C^1$ continuity \sep Nitsche's method 
\end{keyword}
\end{frontmatter}

\section{Introduction}

Computer Aided Design (CAD) is used to create digital models of geometric objects. In a CAD model the object is usually described with the help of two-dimensional curves as well as three-dimensional surfaces and volumes. These free-form curves and surfaces/volumes can be described by means of splines, i.e., piecewise polynomial functions. Such CAD models can be used in many applications, e.g., for simulations based on the Finite Element Method (FEM) or Isogeometric Analysis (IGA), which is considered in this paper. IGA, as introduced in \cite{hughes2005isogeometric}, uses the same spline functions that are used to construct the geometry also for the discretization spaces for the computation of numerical simulations. One advantage of IGA over classical higher-order FEM is that it provides basis functions with high smoothness and high polynomial degree, making it ideal for solving high order partial differential equations (PDEs) over single patches. In most applications, however, the geometries can usually not be described with one patch. Multi-patch domains composed of a collection of several patches, or related concepts, are needed.

In this paper, we restrict ourselves to planar multi-patch domains where the patch parametrizations are matching along the interfaces. While constructing $C^0$ smooth basis functions is quite straightforward, see, e.g.,~\cite{cottrell2009isogeometric, scott2014isogeometric}, imposing higher smoothness in a multi-patch setting is non-trivial. This makes solving higher order equations a more challenging task. In the following, we focus on fourth order problems, such as the biharmonic equation or a Kirchhoff-Love plate or shell formulation.

Two basic ways to get around this problem are to impose the $C^1$ smoothness weakly or strongly, that is, to adjust the variational problem or to construct special basis functions with higher smoothness, respectively. Following the first approach, one way to solve fourth order equations while keeping discontinuous, patch-wise basis functions is to employ a discontinuous Galerkin (dG) discretization as studied in~\cite{moore2018discontinuous,moore2020multipatch}. dG methods approximate the solution with patch-wise defined functions, which are discontinuous across patch interfaces. As a consequence, the variational formulation contains additional integral terms. See also~\cite{nguyen2014nitsche, apostolatos2014nitsche, guo2015nitsche}. In this paper we use a Nitsche formulation for a $C^0$ smooth multi-patch discretization. In that case, a stability term, penalizing the jump of the normal derivative, is added to ensure the coercivity of the bilinear form. Hence, the corresponding stability parameter must be chosen sufficiently large. On the other hand, a too large stability parameter penalizes the jump of the normal derivative too much, which leads to locking of the solution. Therefore, the optimal range for the stability parameter must be determined, which we examine in more detail in this paper.

Another option to circumvent $C^1$ smoothness is to use a mixed/hybrid formulation as in~\cite{zulehner2015ciarlet, rafetseder2018decomposition, rafetseder2019new, pauly2020divdiv}. Here the PDE is reformulated in such a way, by introducing an extra field, that the resulting mixed formulation is of lower order. The obtained formulation has a saddle-point structure. To solve this problem efficiently, one needs a suitable preconditioner. Furthermore, it has a larger system to solve in comparison to the original problem.

At last, $C^1$ smoothness can be enforced weakly by using the mortar method, see, e.g., \cite{brivadis2015isogeometric,horger2019hybrid} for $C^0$-coupling, where the coupling constraints are enforced using Lagrange multipliers, also resulting in a saddle-point problem. In addition to the challenges of the saddle-point structure, there is the difficulty of finding a suitable discrete Lagrange multiplier space such that the resulting formulation is stable.

There are several different strategies when enforcing smoothness over multi-patch domains by strong $C^1$ coupling, i.e., constructing $C^1$ bases along the interfaces. One first attempt is the so-called bending strip method, see~\cite{kiendl2009isogeometric, kiendl2010bending}. More general formulations imposing geometric continuity over multi-patch domains are later developed in~\cite{GrPe15,KaViJu15,collin2016analysis,kapl2017isogeometric}, while the related approaches presented in~\cite{nguyen2014comparative,mourrain_dimension_2016,karvciauskas2016generalizing} follow a more local construction of geometrically continuous splines. See also~\cite{hna2021} for a summary of related approaches. A significant problem of all strong $C^1$ coupling methods are the limitations they pose on the underlying geometries. It was shown in~\cite{collin2016analysis} that a standard isogeometric multi-patch discretization possesses optimal approximation properties only if the parametrization of the domain is a so-called analysis-suitable $G^1$ multi-patch parametrization. Even though many geometries can be reparametrized, cf.~\cite{kapl2018construction}, this is a significant restriction on the geometry. 

It is nonetheless possible to use explicit $C^1$ constructions over more general multi-patch parametrizations by increasing the polynomial degree locally, as in~\cite{chan2018isogeometric,chan2019strong}, and/or by relaxing the smoothness conditions, as in~\cite{WEINMULLER2021114017,takacstoshniwal2022}. In~\cite{WEINMULLER2021114017} basis functions with higher polynomial degree and lower regularity are introduced locally at the interface between two patches, which are not exactly $C^1$ at the interface, but nevertheless yield optimal convergence rates in numerical tests.

In this paper, we focus on two aspects. One is to extend the approximate $C^1$ method introduced in~\cite{WEINMULLER2021114017} to multi-patch domains. Thus, a construction for basis functions at vertices needs to be developed. The second is to compare the approximate $C^1$ method with Nitsche's method, for which we derive conditions on the stability parameter. Thereby we see that the error obtained with the approximate $C^1$ method agrees with the error from Nitsche's method. 

The outline of the paper is as follows: We start with basic notations for B-splines and multi-patch geometries needed for the paper, and introduce the $C^0$ isogeometric multi-patch space. Then the model problem, more precisely, the biharmonic problem is stated and we define two weak formulations, the standard and the Nitsche formulation. In Section~\ref{sec:normal-derivatives} we give the $C^1$ smoothness conditions at the interfaces. Next, we present the basis constructions for the approximate $C^1$ method for the multi-patch geometries and introduce the $C^1$ space used for the approximate $C^1$ method. In Section~\ref{sec:discreteSpaceNitsche} we analyse how the parameter in Nitsche's method must choosen for the method to be stable, which is numerically shown in Section~\ref{sec:numerical-experiments}. There we also compare Nitsche's method with the approximate $C^1$ method on several examples.

\section{Preliminaries}

In this section, we give a brief overview of B-splines and present some notation concerning the multi-patch geometry. We start with the introduction of B-splines and their spaces. After that we present the definition of multi-patch geometries together with the underlying multi-patch topology, which is needed for the basis construction. The section concludes with the description of the $C^0$ isogeometric space.

\subsection{B-splines}

Given positive integers $p$, $r$ and $n$, with $r<p$, and a (uniform) mesh with mesh size $h = 1/n$, the open knot vector $\Xi \coloneqq ( \xi_1,...,\xi_{N+p+1} )$ with $N=p+1+(p-r)(n-1)$ satisfies 
\begin{equation}  
\Xi =
(\underbrace{0,\ldots,0}_{(p+1)-\mbox{\scriptsize times}},
\underbrace{\textstyle h,\ldots ,h}_{(p-r) - \mbox{\scriptsize times}}, 
\underbrace{\textstyle 2h,\ldots ,2h}_{(p-r) - \mbox{\scriptsize times}},\ldots, 
\underbrace{\textstyle (n-1)h,\ldots ,(n-1)h}_{(p-r) - \mbox{\scriptsize times}},
\underbrace{1,\ldots,1}_{(p+1)-\mbox{\scriptsize times}}). \label{eq:knotvector}
\end{equation}
For simplicity, we assume that the regularity is the same at all interior knots and consequently all the knots have the same multiplicity. In general, the knot vector does not need to be uniform and the knot multiplicity may be different for different knots, see also~\cite{WEINMULLER2021114017}. Since we assume uniform knot multiplicity $p-r$, the inter-element continuity is defined by $r$, i.e., we have $C^r$-smoothness at each knot. As usual, the B-spline basis functions $b_i$, with $i = 1,...,N$, can be constructed using the Cox--de Boor recursion, see~\cite{prautzsch2002}. We have
\begin{align*}
\mathcal{S} (p,r,h)  = \text{span} \{ b_i, \; i = 1,...,N \}
\end{align*}
which is the (univariate) spline space of degree $p$, regularity $r$ and mesh size $h$. The spline space consists of functions which are piece-wise polynomials and $C^r$ in $[0,1]$, more precisely,
\begin{equation}
\mathcal{S} (p,r,h)  \coloneqq \{ w \in C^r([0,1]) : w|_{(i h,(i+1) h)} \in \mathbb{P}^p \mbox{ for all }i=0,\ldots,n-1 \}.
\end{equation}
The univariate B-splines can be extended to the two-dimensional case by means of a tensor product structure. The multivariate spline space is defined in the parametric domain $\widehat{\Omega} = [0,1]^2$ by
\begin{align*}
 \boldsymbol{ \mathcal{S} } (\f p, \f r, \f h) = \mathcal{S}_1 (p_1,r_1,h_1) \otimes \mathcal{S}_2 (p_2,r_2,h_2) = \text{span} \{ \boldsymbol{b_{i}} \}_{1 \leq i_1 \leq N_1, \, 1 \leq i_2 \leq N_2},
\end{align*}
where $\f p=(p_1,p_2)$, $\f r=(r_1,r_2)$ and $\boldsymbol{i}=(i_1,i_2)$ denote the corresponding parameter pairs. The tensor-product spline basis function is the product of two univariate basis functions, i.e., $\boldsymbol{b_{i}}(u,v) = b_{i_1}(u)b_{i_2}(v)$ where $u,v \in [0,1]$. We assume throughout the paper that $p_1=p_2=p$, $r_1=r_2=r$ and $h_1=h_2=h$.

\subsection{Multi-patch geometry}\label{sec:multi-patch-geometry}

Let $\Omega$ be a bounded open subset of $\mathbb{R}^2$ with a sufficiently smooth boundary $\partial \Omega$. Moreover, let $\Omega$ be given through a multi-patch segmentation consisting of non-overlapping patches $\Omega^{(k)}$, $k \in \mathcal{M}_P = \{ 1, ..., K\}$, where $K > 0$ is the total number of patches, i.e.,
\[
 \overline{\Omega} = \bigcup_{k\in \mathcal{M}_P} \overline{\Omega^{(k)}},
\]
with $\Omega^{(k)} \cap \Omega^{(l)} = \emptyset$ for all $k\neq l$. Moreover, we assume that no hanging nodes exist. Each patch $\Omega^{(k)}$ is a spline patch with the geometry mapping $\f F^{(k)} \in (\f S^{(k)})^2$ with
\[
 \f F^{(k)} : \widehat{\Omega} \to \overline{\Omega^{(k)}},
\]
where $\f S^{(k)} = \f S( \f p^{(k)}, \f r^{(k)}, \f h^{(k)})$ is a tensor-product spline space of patch $\Omega^{(k)}$. For simplicity, we assume that the spaces in all patches are the same, i.e., $\f p^{(k)} = \f p,\; \f r^{(k)} = \f r$ and $\f h^{(k)} = \f h$ for all $k\in \mathcal{M}_P$.  The construction can also be used for different spline spaces with different degrees, regularities and mesh sizes as long as the interfaces are (partially) matching. We assume that all patch parametrizations $\f F^{(k)}$ are regular, i.e.,
\[
 \det \nabla \f F^{(k)} (u,v) \geq \underline{c} >0, \quad \forall \; (u,v)\in[0,1]^2,
\]
and the closure of $\Omega^{(k)}$ possesses no self-overlaps, i.e., 
\[
 \forall (u,v), (u',v') \in \widehat{\Omega}: (u,v) \neq (u',v')  \Rightarrow \f F^{(k)} (u,v) \neq \f F^{(k)} (u',v') .
\]
In this paper, we introduce a local and a global notation for the mesh objects (i.e.\ edges and vertices) of the multi-patch. While the local notation describes the mesh objects of a single patch, the global notation concerns the relation of edges and vertices on the whole geometry.

Starting with the local setting, each single patch $\Omega^{(k)}$ has four edges $E_s^{(k)}$ and four vertices $V_s^{(k)}$, where $s = 1,\ldots,4$. Following the notation as in Figure~\ref{fig:global_local_notation}, we can describe the edges as follows: $\f F^{(k)}(\widehat{E}_{s}^{(k)}) = E_{s}^{(k)}$ with 
\begin{align*}
 \widehat{E}_{1}^{(k)} &= \{(u, 0)^T, u \in (0,1)\}, \quad
 \widehat{E}_{2}^{(k)} = \{(1,v)^T, v \in (0,1)\}, \\
 \widehat{E}_{3}^{(k)} &= \{(u, 1)^T, u \in (0,1)\}, \quad
 \widehat{E}_{4}^{(k)} = \{(0,v)^T, v \in (0,1)\},
\end{align*}
where $\widehat{E}_{s}^{(k)}$ represent the edges in the parameter setting. For the vertices, we have 
$\f F^{(k)}(\widehat{V}_{s}^{(k)}) = V_{s}^{(k)}$ with 
\begin{align*}
 \widehat{V}_{1}^{(k)} &= (0,0)^T, \quad
 \widehat{V}_{2}^{(k)} = (1,0)^T, \\
 \widehat{V}_{3}^{(k)} &= (1,1)^T, \quad
 \widehat{V}_{4}^{(k)} = (0,1)^T.
\end{align*}
In the global context, the multi-patch consists of several patches $\Omega^{(k)}$ where an edge is either (a) a boundary edge or (b) an interface edge, that is
\begin{align*}
    E_s^{(k)} \ldots \begin{cases}
                \text{boundary edge} & \text{if } E_s^{(k)} \in \partial \Omega, \\
                \text{interface edge} & \text{if } E_s^{(k)} \not\in \partial \Omega \text{ and } \exists l (\neq k) \exists s_l \text{, s.t. } E_s^{(k)} = E_{s_l}^{(l)} .
              \end{cases}
\end{align*}
We introduce the pair $(k,s)$ as a short notation for the index of the edge $E_s^{(k)}$. Then for the interface we introduce the index pair $\kappa = (k,l)$, if there exist $(k,s_k)$, $(l,s_l)$, with $E_{s_k}^{(k)} = E_{s_l}^{(l)}$. Furthermore, we define the set $\mathcal{M}_I$ which collects the ordered pairs of patch-indices for all interfaces and the set $\mathcal{M}_E$ which collects the pairs $(k,s)$, each corresponding to a boundary edge. The sets are defined as
\begin{align*}
 \mathcal{M}_I &= \{ \kappa = (k,l) \; | \; \exists s_k, s_l \in \{1,\ldots,4\}, \text{ s.t. } E_{s_k}^{(k)} = E_{s_l}^{(l)} \text{ and } k < l \}, \\
 \mathcal{M}_E &= \{ \sigma = (k,s) \; | \; \exists s \in \{1,\ldots,4\}, \text{ s.t. } E_s^{(k)} \in \partial \Omega \}.
\end{align*}
In addition, we denote the interface of $\kappa = (k, l) \in \mathcal{M}_I$ with the notation $I_{\kappa}$ just as we denote the boundary edges of $\sigma = (k, s) \in \mathcal{M}_E$ with the notation $E_\sigma$. We assume that for two adjacent patches $\Omega^{(k)}$ and $\Omega^{(l)}$ the patch parametrizations agree along the interface $I_{\kappa}$, summarized in the following.
\begin{ass}[$C^0$-conformity at the interfaces] \label{ass::c0conformityatinterface}
The parametrizations of the two adjacent
patches $k$ and $l$ meet $C^0$ along the interface $I_{\kappa}$, with $\kappa=(k,l)$, i.e., there exists an Euclidean motion $R_{\kappa} : \widehat{E}_{s_k}^{(k)} \rightarrow \widehat{E}_{s_l}^{(l)}$ (a mapping which is a combination of rotation, translation and reflection), such that
  \begin{align*}
  \f F^{(k)}(u,v) = \f F^{(l)} (R_{\kappa}(u,v)) \quad \forall \; (u,v) \in \widehat{E}_{s_k}^{(k)}.
\end{align*}
\end{ass}
Further, there exist three different types of vertices: (a) corner vertices, (b) interface-boundary vertices or (c) inner vertices. We have
\begin{align*}
  V_s^{(k)} \ldots \begin{cases}
                \text{corner vertex} & \text{if } V_s^{(k)} \in \partial \Omega \text{ and } \nexists (l,s_l) (\neq (k,s)) \text{, s.t. } V_s^{(k)} = V_{s_l}^{(l)}, \\
                \text{interface-boundary vertex} & \text{if } V_s^{(k)} \in \partial \Omega \text{ and } \exists l (\neq k)\exists s_l \text{, s.t. } V_s^{(k)} = V_{s_l}^{(l)}, \\
                \text{inner vertex} & \text{if } V_s^{(k)} \not\in \partial \Omega . \\
              \end{cases}
\end{align*}
Similar to the edges, we introduce the set $\mathcal{M}_V$ which collects all (unique) vertices $V_\iota$. Let $\nu$ be the valence of vertex $V_\iota$ and $\Omega^{(k_1)},...,\Omega^{(k_\nu)}$ be the patches around the vertex $V_\iota$. Then we have $V_\iota = V_{s_{k_1}}^{(k_1)} = ... = V_{s_{k_\nu}}^{(k_\nu)}$ and the set is defined with ordered tuples of patch-indices:
\begin{align*}
    \mathcal{M}_V &= \left\{ \iota = (k_1,...,k_\nu) \; | \; \begin{array}{l} 
     V_{s_{k_1}}^{(k_1)} = ... = V_{s_{k_\nu}}^{(k_\nu)} \text{ and } k_\nu > ... > k_1 \mbox{ and } \\
     \nexists (k',s') \notin \{(k_1,s_{k_1}),\ldots,(k_\nu,s_{k_\nu})\} \mbox{ s.t. }  V_{s'}^{(k')} = V_{s_{k_1}}^{(k_1)} 
    \end{array}
    \right\}. 
\end{align*}
Figure~\ref{fig:global_local_notation} gives an example of the nomenclature of the topology for a multi-patch parametrization.

\begin{figure}[h!]
 \centering
 \includegraphics[width=0.9\textwidth]{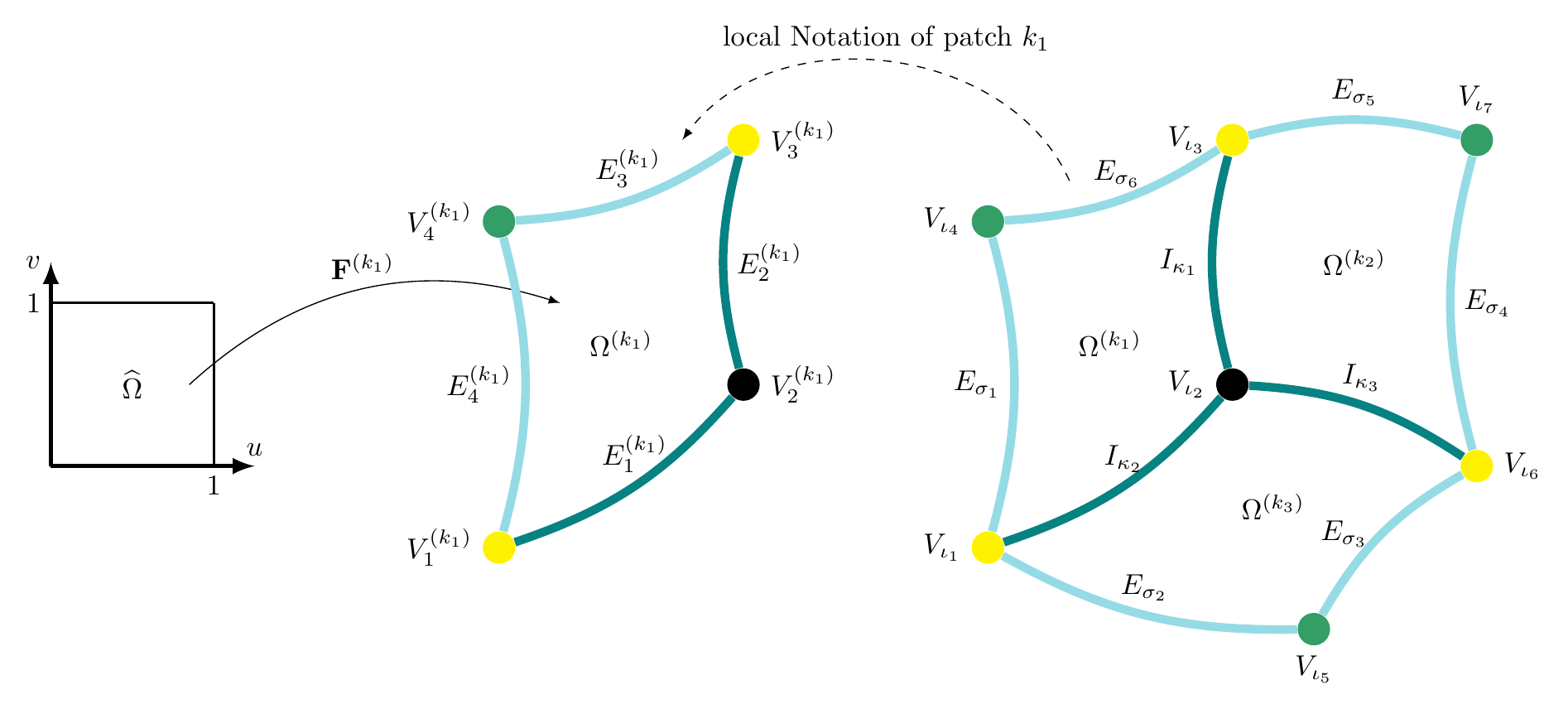}
 \caption{The multi-patch notation: on the right side the global notation is defined and on the left side the local notation of patch $k_1$ where we also have the mapping $\f F^{(k_1)}$ from the parameter domain $\widehat{\Omega} = [0,1]^2$ to the patch $\overline{\Omega^{(k_1)}}$. We start with the local notation: each patch has four edges and four vertices. As an example, the local notation of patch $k_1$ is shown in the figure. Its local edges are indicated by $E_s^{(k_1)}$ and its vertices by $V_s^{(k_1)}$, where $s \in \{ 1,...,4 \}$. In the global setting, the given geometry is constructed with three patches, indicated with $\Omega^{(k_1)}$, $\Omega^{(k_2)}$ and $\Omega^{(k_3)}$. Furthermore the geometry has three interface edges, labeled with $I_{\kappa}, \; \kappa \in \{ 1,...,3 \} $, six boundary edges, denoted as $E_\sigma, \; \sigma \in \{ 1,...,6 \} $ and seven vertices, $V_\iota, \; \iota \in \{ 1,...,7 \} $. From the topology of the geometry, we see that $E_{1}^{(k_1)}$ and $E_{2}^{(k_1)}$ are interface edges and $E_{3}^{(k_1)}$ and $E_{4}^{(k_1)}$ are boundary edges. Furthermore, the vertex $V_{2}^{(k_1)}$ is an interior vertex, the vertices $V_{1}^{(k_1)}$ and $V_{3}^{(k_1)}$ are interface-boundary vertices and the vertex $V_{4}^{(k_1)}$ a corner vertex.} \label{fig:global_local_notation}
\end{figure}

\subsection{$C^0$ isogeometric multi-patch spaces}

In this section we define isogeometric function spaces over multi-patch domains. An isogeometric function $\varphi:\Omega\rightarrow \mathbb{R}$ is defined such that
\[
 \varphi \circ \f F^{(k)} \in \f S^{(k)}.
\]
For this definition to be consistent at the patch-interfaces, we assume global $C^0$-smoothness, resulting in the following definition. The $C^0$ \emph{isogeometric multi-patch space} on $\Omega$ is given as
\begin{equation}
  \mathcal{X}_h = \{ \varphi \in C^0 (\Omega) \text{ such that } w^{(k)} = \varphi \circ \f F^{(k)} \in \f S^{(k)} \; \forall \, k \in \mathcal{M}_P \}.\label{eq:space-X_h}
\end{equation}
Since the patch parametrizations are meeting $C^0$ at the interface, see Assumption~\ref{ass::c0conformityatinterface}, and the knot vectors are assumed to be uniform, see~\eqref{eq:knotvector}, the meshes are conforming along the interfaces. Thus we achieve $C^0$-smoothness at the interfaces if the corresponding spline coefficients at both sides of the interface are equal. Therefore a basis for the isogeometric $C^0$ multi-patch space $\mathcal{X}_h$ can be constructed easily, cf.~\cite{beiraodaveiga2014mathematical} for a more detailed description.

\section{The model problem and its weak formulations}

As a model problem we consider the biharmonic equation. We first set up the equation with two different combinations of boundary conditions. Then we present the weak formulation in the continuous setting, as well as in a discontinuous setting, following Nitsche's method. The section concludes with the two discrete formulations, a Nitsche discretization having discontinuous and a strong discretization having approximately continuous derivatives across interfaces.

\subsection{The model problem}

Let $\Omega$ be a multi-patch geometry as in Section~\ref{sec:multi-patch-geometry} and let $f:\Omega\rightarrow \mathbb{R}$ be a given source funtion on $\Omega$. The biharmonic equation is given as
\begin{align}
  \Delta^2 \varphi = f \quad \text{ in } \Omega \label{eq:biharmonic}
\end{align}
with the boundary conditions
\begin{align}
  \left.
  \begin{array}{ll}
   \varphi &= g_0 \\
   \partial_{\f n} \varphi &= g_1
  \end{array}\right\}
   \quad &\text{ on } \Gamma_N \quad \text{ and } \label{eq:boundary2} \\
  \left.
  \begin{array}{ll}
   \varphi &= g_0 \\
   \Delta \varphi &= g_2
  \end{array}\right\}
   \quad &\text{ on } \Gamma_L, \label{eq:boundary3} 
\end{align}
where $\overline{\Gamma_N \cup \Gamma_L} = \partial \Omega$, $\Gamma_N \cap \Gamma_L = \emptyset$ and $\f n $ is the unit outward normal vector on $\partial \Omega$. The functions $g_0$, $g_1$ and $g_2$ are all assumed to be sufficiently smooth. The boundary conditions concerning the function value and normal derivative can be imposed as essential boundary conditions, hence the problem can be homogenized, see also Section~\ref{sec:imposing-bc}. So we assume from now on, that the problem is already homogeneous, that is, $g_0 = g_1 = 0$. 

\subsection{The standard weak formulation}

We introduce the space
\[
 \mathcal{V}_0 \coloneqq \{ \varphi \in H^2 (\Omega) \, | \, \varphi = 0 \text{ on } \partial \Omega \text{ and } \partial_{\f n} \varphi = 0 \text{ on } \Gamma_N \},
\]
where for any $m \in \mathbb{N}$ and for any open, sufficiently smooth domain $D$ the space $H^m (D)$ is the standard Sobolev space over $D \subset \mathbb{R}^2$ with the standard scalar product 
\[
 (\varphi, \psi)_{H^m(D)} \coloneqq \int_{D} \sum_{i+j=m} (\partial_x^i \partial_y^j \varphi )(\partial_x^i \partial_y^j \psi ) \,\mathrm{d} \f x
\]
and norm
\[
 \| \varphi \|_{H^m(D)} \coloneqq \left( \sum_{j=0}^m (\varphi, \varphi)_{H^j(D)} \right)^{1/2}.
\]
We denote by $(\cdot, \cdot)_{H^0(D)} = (\cdot, \cdot)_{L^2(D)}$ the scalar product of the standard Lebesgue space $L^2 (D)$ with the norm $\norm{\cdot}_{L^2(D)}$. The weak formulation of problem (\ref{eq:biharmonic})-(\ref{eq:boundary3}) is the following.
\begin{problem}\label{problem:model-problem}
Find $\varphi \in \mathcal{V}_0$ such that
\begin{align*}
  (\Delta \varphi, \Delta \psi)_{L^2(\Omega)} = ( f,\psi )_{L^2(\Omega)} + (g_2, \partial_{\f n} \psi)_{L^2(\Gamma_L)} \qquad \forall \; \psi \in \mathcal{V}_0,
\end{align*}
where $\partial_{\f n}$ is the normal derivative at the boundary.
\end{problem}

\subsection{The Nitsche formulation}

Since the geometry is given by a collection of subdomains $\Omega = \bigcup_k \Omega^{(k)}$, a discontinuous Galerkin-type approach is a natural alternative to a fully conforming discretization. Therefore, we introduce for each $m \in \mathbb{N}$ the broken Sobolev space
\[
  \mathcal{H}^m (\Omega) \coloneqq \{ \varphi \in H^1 (\Omega) \; | \; \varphi|_{\Omega^{(k)}} \in H^m (\Omega^{(k)}) \}
\]
with the norm and inner product
\begin{align*}
  \norm{\varphi}_{\mathcal{H}^m (\Omega)} \coloneqq \left( \sum_{i=0}^m (\varphi,\varphi)_{\mathcal{H}^i (\Omega)} \right)^{1/2} \quad \text{ and } \quad (\varphi,\psi)_{\mathcal{H}^m (\Omega)} \coloneqq \sum_{k\in \mathcal{M}_P} (\varphi,\psi)_{H^m (\Omega^{(k)})}.
\end{align*}
 Moreover, we introduce the discretization space
\begin{align}
  \mathcal{X}_0 &:= \{ \varphi \in \mathcal{H}^2 (\Omega) \; | \; \varphi = 0 \text{ on } \partial \Omega \text{ and } \partial_{\f n} \varphi = 0 \text{ on } \Gamma_N \}.\label{eq:space-X_0}
\end{align}
Let $\average{\circ}_{\kappa} = \frac{1}{2}(\circ^{(l)}+\circ^{(k)})|_{I_{\kappa}}$ denote the average and $\jump{\circ}_{\kappa} =(\circ^{(l)}-\circ^{(k)})|_{I_{\kappa}}$ denote the jump across the interface~$I_{\kappa}$.
Since $\mathcal{V}_0 = H^2 (\Omega) \cap \mathcal{X}_0$ we have $\jump{\varphi}_{\kappa} = 0$ for any function $\varphi \in \mathcal{V}_0$ and for all interfaces $I_{\kappa}$.

We follow Nitsche's method, that is, we reformulate the model problem following a symmetric interior penalty Galerkin approach. Thus, we consider the following problem, cf.\ in~\cite{moore2018discontinuous,moore2020multipatch}.
\begin{problem} \label{problem:nitscheformulation}
 Find $\varphi \in \mathcal{X}_0$ such that
 \begin{align*}
    (\varphi,\psi)_{A_h} = ( f,\psi )_{L^2(\Omega)} + (g_2, \partial_{\f n}  \psi)_{L^2(\Gamma_L)}  \quad \forall \psi \in \mathcal{X}_0,
 \end{align*}
where
\begin{align*}
  (\varphi,\psi)_{A_h} &= (\Delta\varphi, \Delta\psi)_{\mathcal{H}^0(\Omega)} - (\varphi, \psi)_{B_h} - (\psi, \varphi)_{B_h} + (\varphi, \psi)_{C_h} \\
   (\varphi, \psi)_{B_h} &= \sum_{\kappa \in \mathcal{M}_I} ( \jump{\partial_{\f n_{\kappa}} \varphi}_{\kappa}, \average{\Delta \psi}_{\kappa} )_{L^2(I_{\kappa})} \\
   (\varphi, \psi)_{C_h} &= \sum_{\kappa \in \mathcal{M}_I} \frac{\eta_{\kappa}}{h} ( \jump{\partial_{\f n_{\kappa}} \varphi}_{\kappa}, \jump{\partial_{\f n_{\kappa}} \psi}_{\kappa} )_{L^2(I_{\kappa})}.
\end{align*}
where $\eta_{\kappa} > 0$ is a prescribed stability parameter assigned to each interface $I_{\kappa}$.
\end{problem}
Conditions on the stability parameters $\eta_{\kappa}$ are discussed in Section~\ref{sec:discreteSpaceNitsche}.

\subsection{The discrete formulations}

In this subsection, we formulate two discrete problems. We have introduced the isogeometric space $\mathcal{X}_{h}\subset \mathcal{H}^2 (\Omega)$ in~\eqref{eq:space-X_h}. Thus, it yields a suitable discretization space~$\mathcal{X}_{h,0}=\mathcal{X}_{h} \cap \mathcal{X}_{0}$ for Problem~\ref{problem:nitscheformulation}. However, the space $\mathcal{V}_{h,0} = \mathcal{X}_{h}\cap \mathcal{V}_{0} = \mathcal{X}_{h}\cap C^1 (\Omega)$ is in general too restrictive, cf.~\cite{collin2016analysis}. Hence, we introduce a different space~$\widetilde{\mathcal{V}}_{h,0} \neq \mathcal{V}_{h,0}$ to  discretize Problem~\ref{problem:model-problem}. The construction of this space is described in Section~\ref{sec:construction_of_discrete_space}. While the first space $\mathcal{X}_{h,0}$ fulfills the required conformity $\mathcal{X}_{h,0} \subset \mathcal{X}_{0}$ by definition, the second space does not in general fulfill the conformity relation~$\widetilde{\mathcal{V}}_{h,0} \nsubseteq \mathcal{V}_{0}$. The reason for this is that the space is spanned with basis functions that are not $C^1$ at the interfaces, but only approximately $C^1$. However, in the limit the jump of the normal derivative across the interface vanishes by construction, see~\cite{WEINMULLER2021114017}. Therefore, for the discretization we treat the functions from $\widetilde{\mathcal{V}}_{h,0}$ as if they were $C^1$ at the interfaces. As a consequence, the additional terms vanish in the variational formulation and no interface integrals need to be calculated.
Furthermore, in some special cases, we achieve exact $C^1$ smoothness at the interface and we have $\widetilde{\mathcal{V}}_{h,0} \subset H^2(\Omega)$ which is described in more detail in Remark~\ref{rmk:exactC1smoothness}. 

Discretizing Problem~\ref{problem:model-problem} using~$\widetilde{\mathcal{V}}_{h,0}$, we obtain the following discrete problem.
\begin{problem} \label{problem:approxC1discreteformulation}
 Find $\varphi_h \in \widetilde{\mathcal{V}}_{h,0}$ such that
 \begin{align}
  (\Delta\varphi_h,\Delta\psi_h)_{\mathcal{H}^0(\Omega)} =  ( f,\psi_h )_{L^2(\Omega)} + (g_2, \partial_{\f n}  \psi_h)_{L^2(\Gamma_L)} \quad \forall \psi_h \in \widetilde{\mathcal{V}}_{h,0}.
 \end{align}
\end{problem}
Discretizing Problem~\ref{problem:nitscheformulation} using the space~$ \mathcal{X}_{h,0}$, we obtain the following discrete problem.
\begin{problem} \label{problem:discreteformulationnitsche}
 Find $\varphi_h \in \mathcal{X}_{h,0}$ such that
 \begin{align}
  (\varphi_h,\psi_h)_{A_h} = ( f,\psi_h )_{L^2(\Omega)} + (g_2, \partial_{\f n}  \psi_h)_{L^2(\Gamma_L)}  \quad \forall \psi_h \in \mathcal{X}_{h,0},
 \end{align}
\end{problem}
 As pointed out before, the space $\mathcal{V}_{h,0} = \mathcal{X}_{h} \cap \mathcal{V}_{0} \subset H^2(\Omega)$ is not a suitable discretization space for Problem~\ref{problem:approxC1discreteformulation}, since its approximation power is in general drastically reduced. In Section~\ref{sec:numerical-experiments} we compare with numerical experiments the solutions of Problem~\ref{problem:discreteformulationnitsche} and Problem~\ref{problem:approxC1discreteformulation}.

\section{Normal derivatives and $C^1$ smoothness conditions at interfaces}
\label{sec:normal-derivatives}

In order to solve fourth order problems on multi-patch domains, we need to give a description of the normal derivative of an isogeometric function across an interface. This is necessary both for the definition of the bilinear forms $(\cdot, \cdot)_{B_h}$ and $(\cdot, \cdot)_{C_h}$ as well as for the definition of the isogeomtric space $\widetilde{\mathcal{V}}_{h,0}$, which we develop in detail in Section~\ref{sec:construction_of_discrete_space}. Let us focus  first on one edge of a patch $\Omega^{(k)}$, without loss of generality we consider the edge with $u=0$, that is,
\begin{align*}
  E^{(k)}_4 = \{ \f F^{(k)} (0,v) : \; v \in [0,1] \}.
\end{align*}
We now define the tangential derivative along the edge to be
\[
  \f t(v) := \partial_v \f F^{(k)} (0,v),
\]
and the unit tangent vector
\[
  \f t_0 (v) := \frac{\f t (v)}{\tau (v)},
\]
where $\tau (v) = \norm{\f t (v)}$. We define the outer normal vector of $\partial \Omega^{(k)}$ to be $\f n_k$, with
\begin{align}
 \f n_k \circ \f F^{(k)} (0, v) &= a_1^{(k)} (v) \partial_u \f F^{(k)} (0,v) + a_2^{(k)} (v) \partial_v \f F^{(k)} (0,v) \label{eq:exactnormalvector}
\end{align}
where $a_1^{(k)}, a_2^{(k)}$ are functions given as
\begin{align*}
 a_1^{(k)} (v) = -\frac{1}{\alpha^{(\kappa,k)} (v)} \quad \text{ and } \quad a_2^{(k)}(v) = \frac{ \beta^{(\kappa,k)} (v)}{\alpha^{(\kappa,k)} (v)},
\end{align*}
with
\begin{align}
 \left. \begin{array}{r l}
 \alpha^{(\kappa,k)} (v) &= \det \left( \partial_u \f F^{(k)} (0,v), \f t(v) \right), \\
 \beta^{(\kappa,k)} (v) &= \frac{\partial_u \f F^{(k)} (0,v) \cdot \f t_0(v)}{\tau (v)}, \end{array}  \right. \label{def::alpha_beta}
\end{align}
following the proof of~\cite[Proposition 2]{collin2016analysis}.

Given a function $\varphi:\Omega\rightarrow \mathbb{R}$, with $\varphi\circ \f F^{(k)} = f^{(k)}$, the normal derivative of $\varphi$ along the edge can be described by
\begin{align}
 \left( \partial_{\f n_k} \varphi \right) \circ \f F^{(k)} = \f n_k \cdot \left((\nabla_{\boldsymbol{x}} \varphi) \circ \f F^{(k)}\right) &= (a_1^{(k)} , a_2^{(k)} ) \nabla f^{(k)} \\
 &= -\frac{1}{\alpha^{(\kappa,k)} (v)} \left( \partial_u f^{(k)} (0,v) -\beta^{(\kappa,k)} (v) \partial_v f^{(k)} (0,v) \right), \label{eq:normal-derivative-one-edge}
\end{align}
since
\[
 \nabla \f F^{(k)} \nabla_{\boldsymbol{x}} \varphi = \nabla f^{(k)}.
\]
Here $\nabla_{\boldsymbol{x}}$ denotes the gradient in physical space and $\nabla$ denotes the gradient in $(u,v)$-coordinates, in particular,
\[
 \nabla \f F^{(k)} = \left( \partial_u \f F^{(k)} , \partial_v \f F^{(k)} \right)^T.
\]  
Hence, a patch-wise defined function $\varphi$ is $C^1$-smooth along the interface $I_{\kappa}$ between $\Omega^{(k)}$ and $\Omega^{(l)}$, iff
\begin{align}
 \partial_{{\f n}_k} \varphi |_{E_{s_k}^{(k)}} = - \partial_{{\f n}_l} \varphi |_{E_{s_l}^{(l)}} , \label{eq:C1-one-interface}
\end{align}
Here we consider the restriction to $E_{s_k}^{(k)}$ to be the limit from the side $\Omega^{(k)}$, whereas the restriction to $E_{s_l}^{(l)}$ denotes the limit from $\Omega^{(l)}$.
\begin{rmk}
Let $s_k=4$ and $s_l=1$, i.e., 
\[
 \f F^{(k)} (0, t) = \f F^{(l)} (t, 0),
\]
and
\[
  \f n_{\kappa} = \f n_k = - \f n_l,
\]
we obtain from \eqref{eq:C1-one-interface} and \eqref{eq:normal-derivative-one-edge} the following $C^1$ condition for the pull-backs,
\begin{align*}
     -\frac{1}{\alpha^{(\kappa,k)} (t)} \left( \partial_u f^{(k)} (0,t) -\beta^{(\kappa,k)} (t) \partial_v f^{(k)} (0,t) \right) = \frac{1}{\alpha^{(\kappa,l)} (t)} \left( \partial_v f^{(l)} (t,0) - \beta^{(\kappa,l)} (t) \partial_u f^{(l)} (t,0) \right),
\end{align*}
where
\begin{align*}
 \left. \begin{array}{r l}
 \alpha^{(\kappa,l)} (t) &= \det \left( \f t(t) , \partial_v \f F^{(l)} (t,0) \right), \\
 \beta^{(\kappa,l)} (t) &= \frac{\f t_0(t) \cdot \partial_v \f F^{(l)} (t,0)}{\tau (t)}. \end{array}  \right.
\end{align*}
In general, the isogeometric function is $C^1$, if its graph parametrization is $G^1$, which is determined by the relation
\begin{align*}
  \alpha^{(\kappa,k)} (t) \begin{bmatrix}
                           \partial_v \f F^{(l)} (t,0) \\
                           \partial_v f^{(l)} (t,0)
                          \end{bmatrix}
   &+  \alpha^{(\kappa,l)} (t) \begin{bmatrix}
                               \partial_u \f F^{(k)} (0,t) \\
                               \partial_u f^{(k)} (0,t)
                              \end{bmatrix}
    - \\ &\left( \alpha^{(\kappa,k)} (t) \beta^{(\kappa,l)} (t) + \alpha^{(\kappa,l)} (t) \beta^{(\kappa,k)} (t) \right) \begin{bmatrix}
                            \partial_v \f F^{(k)} (0,t) \\
                            \partial_v f^{(k)} (0,t)
                            \end{bmatrix}
     = \boldsymbol{0}.
\end{align*}
For a further discussion on the statement above, see, e.g.,~\cite{GrPe15,KaViJu15}. Due to this relation, the functions $\alpha^{(\kappa,k)}, \alpha^{(\kappa,l)}$ and $\beta^{(\kappa,k)}, \beta^{(\kappa,l)}$ are called gluing data. When constructing basis functions related to boundary edges, we introduce artificial gluing data, by setting $\alpha^{(k,s)} \equiv 1$ and $\beta^{(k,s)} \equiv 0$ iff the edge $E^{(k)}_s$ is a boundary edge.
\end{rmk}

\section{The construction of the discrete space for the approximate $C^1$ method} \label{sec:construction_of_discrete_space}

In this section, we explain the basis construction for the space $\widetilde{\mathcal{V}}_{h,0}$ used in Problem~\ref{problem:approxC1discreteformulation}. To do this, we first describe the basis functions in the local setting, i.e., patch-wise, and then define the global basis functions by gluing the functions together at the interfaces and at the vertices.

In the local setting, we introduce three different types of subspaces: the patch interior, the edge and the vertex space. Each patch $\Omega^{(k)}$ can be described by these spaces, more precisely, each patch can be divided into nine subspaces: one patch interior, four edge and four vertex spaces, corresponding to the topology of the patch geometry. Each subspace (interior, edge and vertex space) is spanned by different basis functions. Let $\Omega^{(k)}$ be the patch in which we have the interior space $\mathcal{A}_\circ^{(k)}$, the four edge spaces $\mathcal{A}_{E,s}^{(k)}$ and the four vertex spaces $\mathcal{A}_{V,s}^{(k)}$, $s =1,...,4$. 
In Subsection~\ref{sec:theninesubspaces} we describe the construction in the parameter domain. Figure~\ref{fig:ninesubspaces} shows an overview of the local patch-wise separation.

\begin{figure}[h!]
 \centering
 \includegraphics[width=\textwidth]{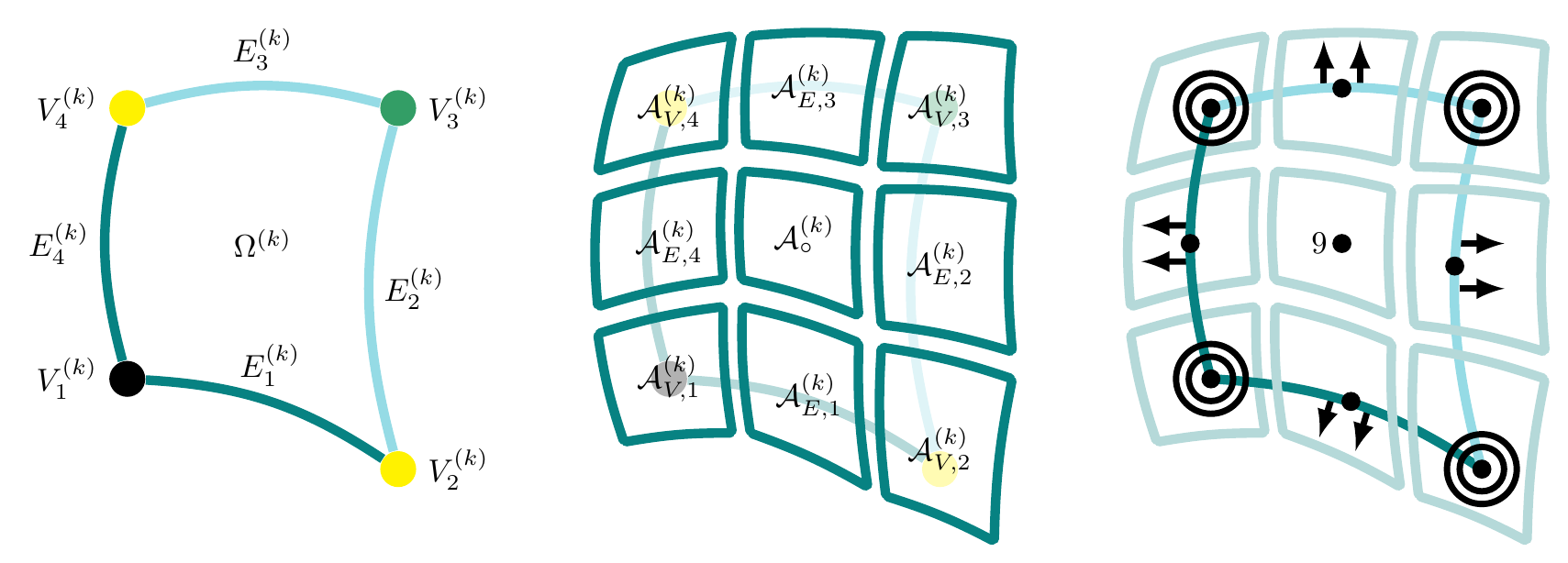}
 \caption{The nine subspaces of patch $k$: one interior, four edge and four vertex spaces. In the right figure we give an example for the dofs. Thus, we choose the knot vector $\Xi = (0, 0, 0, 0, 1/4,1/2,3/4,1,1,1,1)$ for both directions to obtain the number of dofs. Follow the construction in Subsection~\ref{sec:theninesubspaces}, we have nine interior dofs, three for each edge and six for each vertex. Similar to the element notation in FEM, the arrows at the edges represent the dofs for the normal derivative at the edge and the circles at the vertices the dofs for the first and second derivatives at the vertices.} \label{fig:ninesubspaces}
\end{figure}

To obtain the global basis functions, we need to match the basis functions at the interfaces and vertices to ensure the $C^0$ and the approximate $C^1$ continuity. This procedure is explained in Subsection~\ref{sec:coc1coupling}. The resulting spaces are defined as $\mathcal{A}_{I_\kappa}$ and $\mathcal{A}_{V_\iota}$. Further, we denote the boundary space as $\mathcal{A}_{B_\sigma}$ for each boundary edge. As the result, we have
\[
 \widetilde{\mathcal{V}}_{h} = \left( \bigoplus_{k \in \mathcal{M}_P} \mathcal{A}_\circ^{(k)} \right) \oplus \left( \bigoplus_{\kappa \in \mathcal{M}_I} \mathcal{A}_{I_\kappa} \right) \oplus \left( \bigoplus_{\iota \in \mathcal{M}_V} \mathcal{A}_{V_\iota} \right) \oplus \left( \bigoplus_{\sigma \in \mathcal{M}_E} \mathcal{A}_{B_\sigma} \right)
\]
and
\[
 \widetilde{\mathcal{V}}_{h,0} = \widetilde{\mathcal{V}}_{h} \cap \mathcal{X}_0.
\]

\subsection{The patch-local subspaces} \label{sec:theninesubspaces}

In this subsection we explain how the pull-backs $\widehat{\mathcal{A}}_\circ^{(k)}$, $\widehat{\mathcal{A}}_{E,s}^{(k)}$ and $\widehat{\mathcal{A}}_{V,s}^{(k)}$ of the different spaces ${\mathcal{A}}_\circ^{(k)}$, ${\mathcal{A}}_{E,s}^{(k)}$ and ${\mathcal{A}}_{V,s}^{(k)}$ are constructed. We visualize them with the help of an example as shown in Figure~\ref{fig:ninesubspaces}. There we choose the knot vector $\Xi = (0, 0, 0, 0, 1/4,1/2,3/4,1,1,1,1)$ in both directions. Before we can explain the construction in detail, we need to introduce the concept of the approximated gluing data in the following subsection.

\subsubsection{The approximated gluing data}

As shown in~\cite{kapl2019argyris}, a basis of the (exactly) $C^1$-smooth isogeometric space can be constructed from the gluing data, which appear linearly in the formula similar to~\eqref{eq:taylorexpansion}. The gluing data $\alpha^{(\kappa,k)}$ and $\beta^{(\kappa,k)}$ defined in~\eqref{def::alpha_beta} are in general rather complex. While the functions $\alpha^{(\kappa,k)}$ are splines from $\mathcal{S}(2p-1, r-1,h)$, the functions $\beta^{(\kappa,k)}$ are even piecewise rational functions with regularity $r-1$. Thus, this results in the pull-back of the isogeometric function being a non-trivial rational function. To obtain ``nicer'' basis functions, that is, piecewise polynomials of a controlable, bounded degree, we introduce the approximated gluing data as spline functions, which are computed by a projection into $\mathcal{S} (\widetilde{p}, \widetilde{r}, h)$ with the operator $P_h$, that is,
\[
  \widetilde{\alpha}^{(\kappa,k)} := P_h (\alpha^{(\kappa,k)}) \quad \text{and} \quad   \widetilde{\beta}^{(\kappa,k)} := P_h (\beta^{(\kappa,k)}).
\]
For the numerical experiments, we fix the spline parameters of the approximated gluing data to $\widetilde{p} = p - 1$ and $\widetilde{r} = \widetilde{p} - 1$ to obtain optimal convergence rates, in accordance with~\cite{WEINMULLER2021114017}. One can use a higher polynomial degree and/or lower regularity to approximate the gluing data to improve the approximation of the $C^1$ continuity (or even in some special cases to obtain an exactly $C^1$-smooth space), but this does not improve the results in the numerical experiments, see~\cite{WEINMULLER2021114017}.
Similar to~\eqref{eq:exactnormalvector}, we can express the approximate normal vector $\widetilde{\f n}_k$ for the patch $\Omega^{(k)}$ on the edge $E^{(k)}_4$ ($u=0$) as 
\begin{align}
 \widetilde{\f n}_k \circ \f F^{(k)} (0, v) &= \widetilde{a}_1^{(k)} (v) \partial_u \f F^{(k)} (0,v) + \widetilde{a}_2^{(k)} (v) \partial_v \f F^{(k)} (0,v) \label{eq:approxnormalvector}
\end{align}
where the functions in the linear combination are given as
\begin{align*}
 \widetilde{a}_1^{(k)} = -\frac{1}{\widetilde{\alpha}^{(\kappa,k)}} \quad \text{ and } \quad \widetilde{a}_2^{(k)} = \frac{ \widetilde{\beta}^{(\kappa,k)}}{\widetilde{\alpha}^{(\kappa,k)}}.
\end{align*}

\subsubsection{$C^1$ expansion along one edge}

For the construction of the basis functions, we use a Taylor expansion of the trace and of the transversal derivative as stated in~\cite[Proposition~5]{kapl2017dimension}. For simplicity, we consider again the edge $\widehat{E}^{(k)}_4$ on $\Omega^{(k)}$, as in Section~\ref{sec:normal-derivatives}. Then the functions are defined for all $u,v \in [0,1]$ by 
\begin{align}
f^{(k)}_4 [b^+,b^-] (u,v) = b^+ (v) \left( b_{1} (u) + b_{2} (u) \right) + \left( \widetilde{\alpha}^{(\kappa, k)} (v) b^- (v) + \widetilde{\beta}^{(\kappa, k)} (v) (b^+)'(v) \right) \frac{h}{p} b_{2} (u), \label{eq:taylorexpansion}
\end{align}
where $b^+ \in \mathcal{S}^+ = \mathcal{S} (p, p-1, h)$ and $b^- \in \mathcal{S}^- = \mathcal{S} (p-1,p-2,h)$. Representations $f^{(k)}_s [b^+,b^-]$ for the functions along other edges $\widehat{E}^{(k)}_s$ are defined equivalently. The choice of the spaces $\mathcal{S}^+$ and $\mathcal{S}^-$ is derived from~\cite[Corollary~7]{kapl2017dimension} for AS-$G^1$ geometries, which also turned out to be the ideal choice to obtain optimal convergence rates in the numerical tests on general geometries, cf.~\cite{WEINMULLER2021114017}. We have by construction
\[
 f^{(k)}_4 [b^+,b^-]  \in \mathcal{S}_1 (p, r, h ) \otimes \mathcal{S}_2 (p + \widetilde{p} - 1, \min \{ \widetilde{r}, r,p-2 \}, h ).
\]
While the first variable in the function, i.e., $b^+$, describes the trace at the edge, the second variable, i.e., $b^-$, specifies the directional derivative in the direction of the approximate normal vector $\widetilde{\f n}_k$. Hence, we have for $\varphi_h|_{\Omega^{(k)}} = f^{(k)}_s [b^+,b^-] \circ (\f F^{(k)})^{-1}$ that
\begin{align}
  \varphi_h \circ \f F^{(k)} |_{{\widehat{E}}_s^{(k)}} &= b^+, \\
  \partial_{\widetilde{\f n}_k} \varphi_h \circ \f F^{(k)} |_{{\widehat{E}}_s^{(k)}} &= -b^-. 
\end{align}
The term $\partial_{\widetilde{\f n}_k} \varphi_h$ is an approximation of the normal derivative at the edge due to the definition of the approximated gluing data, see~\cite[Proposition~4]{WEINMULLER2021114017}. 
\begin{lem}
For the interface $I_\kappa$, if the pull-back of $\varphi_h$ is given by $f^{(k)}_{s_k}[b^+,b^-]$ on $\Omega^{(k)}$ and by $f^{(l)}_{s_l}[b^+,b^-]$ on $\Omega^{(l)}$, then using~\eqref{eq:approxnormalvector} and~\eqref{eq:taylorexpansion} gives us 
\begin{align*}
 \partial_{\widetilde{\f n}_k} \varphi_h|_{E_{s_k}^{(k)}} = - \partial_{\widetilde{\f n}_l} \varphi_h|_{E_{s_l}^{(l)}} . \label{eq:approxnormalderiv}
\end{align*}
Recall that $\widetilde{\f n}_k \approx {\f n}_k = - {\f n}_l \approx - \widetilde{\f n}_l$. Here we consider the restriction to $E_{s_k}^{(k)}$ to be the limit from the side $\Omega^{(k)}$, whereas the restriction to $E_{s_l}^{(l)}$ denotes the limit from $\Omega^{(l)}$.
\end{lem}
In other words, we impose an exact coupling of the approximate normal derivatives $\partial_{\widetilde{\f n}_k}$ and $\partial_{\widetilde{\f n}_l}$. To obtain exact $C^1$ smoothness we need $\partial_{\widetilde{\f n}_k} = \partial_{{\f n}_k}$, which is achieved by $\widetilde{\f n}_k = - \widetilde{\f n}_l = \f n$. For further discussions, see Remark~\ref{rmk:exactC1smoothness}. In the case of an boundary edge, we replace the approximate gluing data with $\alpha^{(k,s)} \equiv 1$ and $\beta^{(k,s)} \equiv 0$. 

\subsubsection{The patch interior basis functions}

The patch interior space is defined as
\begin{align}
\widehat{\mathcal{A}}_\circ^{(k)} = \text{span} \{ \boldsymbol{b}^{(k)}_{\boldsymbol{j}} : \boldsymbol{j} \in \mathcal{I}^{(k)}_{\circ} \} 
\end{align}
where
$\mathcal{I}^{(k)}_{\circ} = \{(i_1,i_2)\in\mathbb{Z}^2:3\leq i_1 \leq N-2, \; 3 \leq i_2 \leq N-2\}$
and $\{ \boldsymbol{b}^{(k)}_{\boldsymbol j}\}$ are the basis functions of the tensor-product B-spline space~$\boldsymbol{\mathcal{S}}^{(k)} = \boldsymbol{ \mathcal{S} } (\f p, \f r, \f h)$ of dimension $N \times N$, with $p \geq 2$ and $p-1 \geq r \geq 1$. In Figure~\ref{fig:interiorspace} an example is shown. 
\begin{lem}
Let $\boldsymbol{b}^{(k)}_{\boldsymbol{j}} \in \widehat{\mathcal{A}}_\circ^{(k)}$, then the isogeometric function $\varphi_h|_{\Omega^{(k)}} = \boldsymbol{b}^{(k)}_{\boldsymbol{j}} \circ (\f F^{(k)})^{-1}$ has vanishing traces and normal derivatives at all edges of $\Omega^{(k)}$, that is, 
\begin{align*}
  \varphi_h |_{\partial\Omega^{(k)}}  & = 0, \\
  \partial_{\widetilde{\f n}_k} \varphi_h |_{\partial\Omega^{(k)}}  & = 0.
\end{align*}
Moreover, the patch interior spline space satisfies $\widehat{\mathcal{A}}_\circ^{(k)} \subset C^1([0,1]^2)$.
\end{lem}

\begin{figure}
 \centering
  \begin{subfigure}[b]{0.32\textwidth}
  \includegraphics[width=\textwidth]{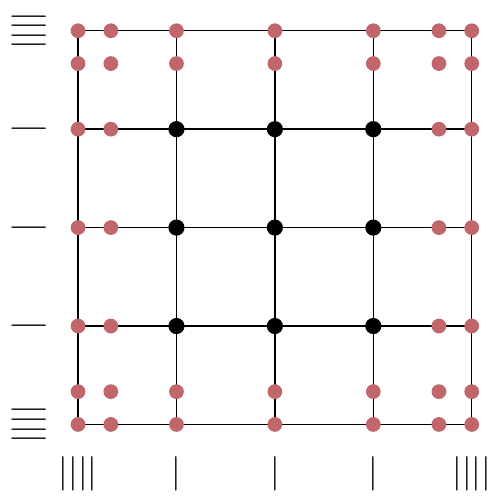}
  \subcaption{The dofs of the patch interior space.}
 \end{subfigure}
 \begin{subfigure}[b]{0.4\textwidth}
  \includegraphics[width=\textwidth]{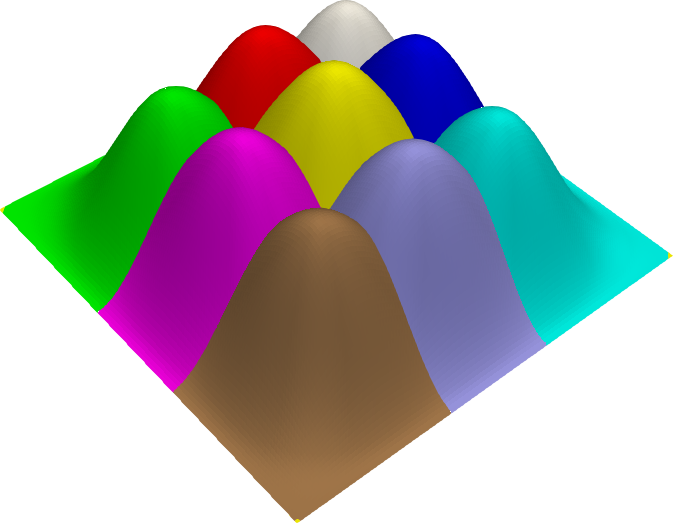}
  \subcaption{The basis functions of the patch interior space.}
 \end{subfigure}
 \caption{The patch interior basis: While the red dots represent the dofs that are eliminated, that is, the corresponding basis functions have non-zero function values or non-zero normal derivatives at the patch boundary, the black dots represent the dofs for the patch interior space.} \label{fig:interiorspace}
\end{figure}

\subsubsection{The edge basis functions}

Without loss of generality, let the edge $E_s^{(k)}$ be such that it corresponds to $u=0$, i.e., $s=4$. Any edge can be rotated and translated in such a way that it coincides with this configuration.

We define the space $\widehat{\mathcal{A}}_{E, s}^{(k)}$ as the span of those basis functions that have non-vanishing trace or approximate normal derivative along the interface and vanishing value, gradient and Hessian at both endpoints of the interface. More precisely, we have 
\begin{align}
\widehat{\mathcal{A}}_{E,s}^{(k)} = \widehat{\mathcal{A}}_{E,s,+}^{(k)} \oplus \widehat{\mathcal{A}}_{E,s,-}^{(k)}, 
\end{align}
with
\begin{align*}
\widehat{\mathcal{A}}_{E,s,+}^{(k)} = \text{span} \{ f^{(k)}_{s} [b_j^+, 0] : 4 \leq j \leq N_+ - 4 \} \quad\mbox{ and } \quad \widehat{\mathcal{A}}_{E,s,-}^{(k)} = \text{span} \{ f^{(k)}_{s} [0, b_j^-] : 3 \leq j \leq N_- -3 \},
\end{align*}
where $f^{(k)}_{s} [\cdot,\cdot]$ is defined as in~\eqref{eq:taylorexpansion}, $\{b_i^+\}$ and $\{b_j^-\}$, $N_+$ and $N_-$ are the bases and dimensions of the spaces $\mathcal{S}^+$ and $\mathcal{S}^-$, respectively. In Figure~\ref{fig:edgespace} we give an example of the edge space.

\begin{lem}
Let $f^{(k)}_{s} [b_j^+, 0] \in \widehat{\mathcal{A}}_{E,s,+}^{(k)}$, then the isogeometric function $\varphi_h|_{\Omega^{(k)}} = f^{(k)}_{s} [b_j^+, 0] \circ (\f F^{(k)})^{-1}$ satisfies 
\begin{align*}
  \varphi_h \circ \f F^{(k)} |_{{\widehat{E}}_s^{(k)}}  & = b_j^+, \\
  \partial_{\widetilde{\f n}_k} \varphi_h \circ \f F^{(k)} |_{{\widehat{E}}_s^{(k)}} &= 0,
\end{align*}
for $f^{(k)}_{s} [0,b_j^-] \in \widehat{\mathcal{A}}_{E,s,-}^{(k)}$ the corresponding isogeometric function satisfies
\begin{align*}
  \varphi_h \circ \f F^{(k)} |_{{\widehat{E}}_s^{(k)}} &= 0, \\
  \partial_{\widetilde{\f n}_k} \varphi_h \circ \f F^{(k)} |_{{\widehat{E}}_s^{(k)}} &= - b_j^-.
\end{align*}
Moreover, all functions $\varphi_h|_{\Omega^{(k)}} = f_h \circ (\f F^{(k)})^{-1}$, with $f_h \in \widehat{\mathcal{A}}_{E,s}^{(k)}$, satisfy
\begin{align*}
  \varphi_h|_{V_s^{(k)}} &= 0, \\
  \nabla \varphi_h|_{V_s^{(k)}} &= \boldsymbol{0}, \\
  \text{Hess} ( \varphi_h )|_{V_s^{(k)}} &= 0_{22},
\end{align*}
for all vertices $V_s^{(k)}$, $s = 1,...,4$.
\end{lem}

\begin{figure}
 \centering
  \begin{subfigure}[b]{0.32\textwidth}
  \includegraphics[width=\textwidth]{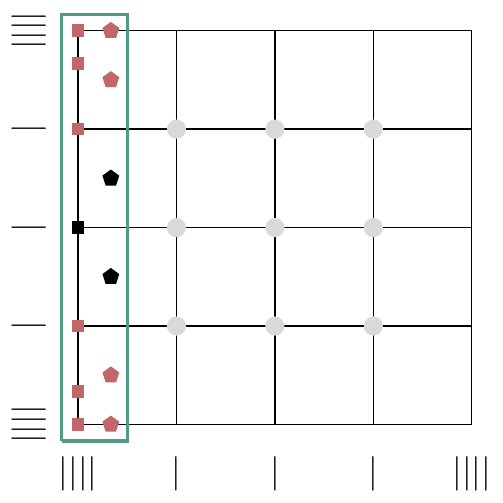}
  \subcaption{The edge basis function.}
 \end{subfigure}
 \begin{subfigure}[b]{0.22\textwidth}
  \includegraphics[width=\textwidth]{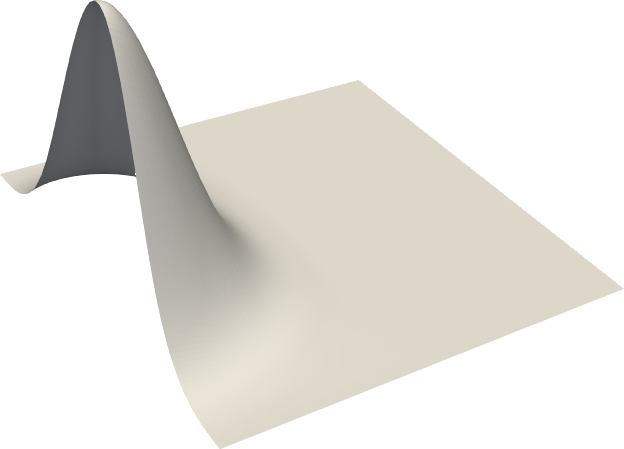}
  \subcaption{$f^{(k)}_{s} [b_4^+, 0]$}
 \end{subfigure}
 \begin{subfigure}[b]{0.22\textwidth}
  \includegraphics[width=\textwidth]{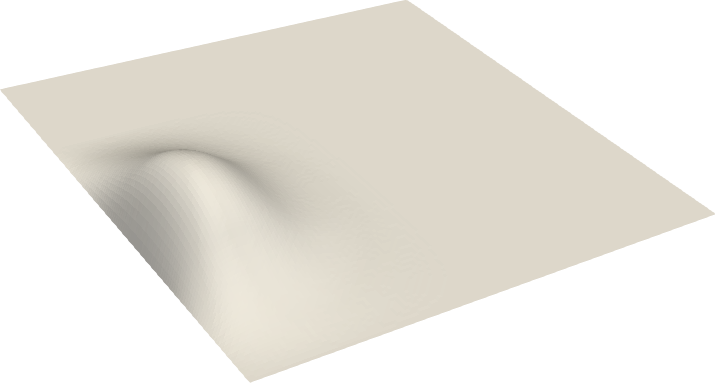}
  \subcaption{$f^{(k)}_{s} [0, b_3^-]$}
 \end{subfigure}
  \begin{subfigure}[b]{0.22\textwidth}
  \includegraphics[width=\textwidth]{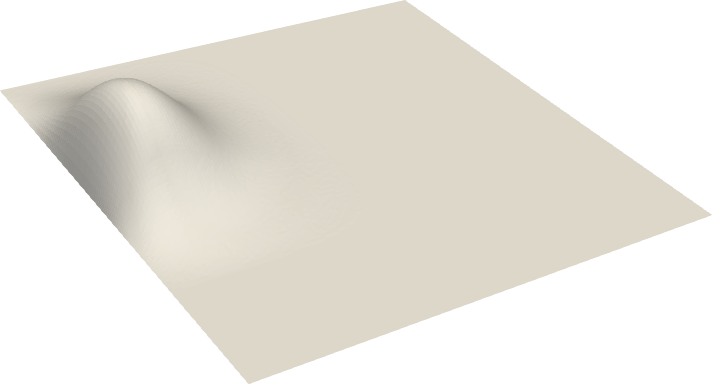}
  \subcaption{$f^{(k)}_{s} [0, b_4^-]$}
 \end{subfigure}
 \caption{We consider here the edge with $u=0$. For each edge we obtain seven trace basis functions depicted with a square and six normal derivative basis functions represented by a pentagon. Then the edge basis functions which have influence up to order two at the vertices are eliminated (marked in red). The remaining three basis functions span the edge space and are shown in Subfigures 4b-4d.} \label{fig:edgespace}
\end{figure}

\subsubsection{The vertex basis functions}

For simplicity of the notation, we assume that the vertex is at $\boldsymbol{x} = V^{(k)}_1 = \f F^{(k)} (0,0)$. Then we collect all the edge basis functions, which have non-vanishing $C^2$-data on one of the two adjacent edges, i.e., which fulfill one of the following conditions
\begin{align*}
  \varphi_h |_{V_1^{(k)}} &\neq 0, \\
  \nabla \varphi_h |_{V_1^{(k)}} &\neq \boldsymbol{0}, \\
  \text{Hess} ( \varphi_h ) |_{V_1^{(k)}} &\neq 0_{22}.
\end{align*}
More precisely, we define three sets of basis functions
\[
 \mathcal{B}_{b.e.} = \{ f_1^{(k)} [b^+_1,0],f_1^{(k)} [b^+_2,0],f_1^{(k)} [b^+_3,0],f_1^{(k)} [0,b^-_1],f_1^{(k)} [0,b^-_2], {\f b}^{(k)}_{(1,3)} \},
\]
corresponding to the bottom edge (with $v=0$),
\[
 \mathcal{B}_{l.e.} = \{ f_4^{(k)} [b^+_1,0],f_4^{(k)} [b^+_2,0],f_4^{(k)} [b^+_3,0],f_4^{(k)} [0,b^-_1],f_4^{(k)} [0,b^-_2], {\f b}^{(k)}_{(3,1)} \},
\]
corresponding to the left edge (with $u=0$), as well as 
\[
 \mathcal{B}_{c.t.} = \{ {\f b}^{(k)}_{(1,1)},{\f b}^{(k)}_{(1,2)},{\f b}^{(k)}_{(1,3)},{\f b}^{(k)}_{(2,1)},{\f b}^{(k)}_{(2,2)}, {\f b}^{(k)}_{(3,1)} \},
\]
which are standard tensor-product B-splines used for constructing a correction term.

To construct the vertex space we perform $C^2$ interpolation at the vertex for all three sets of functions. To do so, we prescribe $C^2$-data $\Phi = \{\Phi_{i_1,i_2}\}_{i_1,i_2}$, with $1\leq i_1, i_2 \leq 3$ and $i_1 + i_2 \leq 4$, in physical space and interpolate the pull-backs using the three spaces defined above. We then add the first two interpolations and subtract the third. The resulting functions are denoted by $g_1^{(k)} [\Phi_{i_1, i_2}]$. We refer to~\ref{sec:appendix-C2-interpolation} for details of the construction. We have by construction
\[
g_s^{(k)} [\Phi_{i_1, i_2}] \in \mathcal{S} (p + \widetilde{p} - 1, \min \{ \widetilde{r}, r,p-2 \}, h ) \otimes \mathcal{S} (p + \widetilde{p} - 1, \min \{ \widetilde{r}, r,p-2 \}, h ).
\]
The space for the vertex $V_s^{(k)}$ is defined as
\begin{align}
\widehat{\mathcal{A}}_{V,s}^{(k)} = \text{span} \{ g_s^{(k)} [\Phi_{i_1, i_2}] : (i_1,i_2) \in \mathcal{I}_{V} \} 
\end{align}
where $\mathcal{I}_{V} = \{ (i_1,i_2)\in\mathbb{Z}^2 : 1\leq i_1, i_2 \leq 3, \text{ and } i_1 + i_2 \leq 4 \}$. Thus, the dimension of the space $\widehat{\mathcal{A}}_{V,s}^{(k)}$ is always six. Figure~\ref{fig:vertexspace} shows an example for the vertex space.

\begin{figure}[h!]
 \centering
  \begin{subfigure}[b]{0.32\textwidth}
  \includegraphics[width=\textwidth]{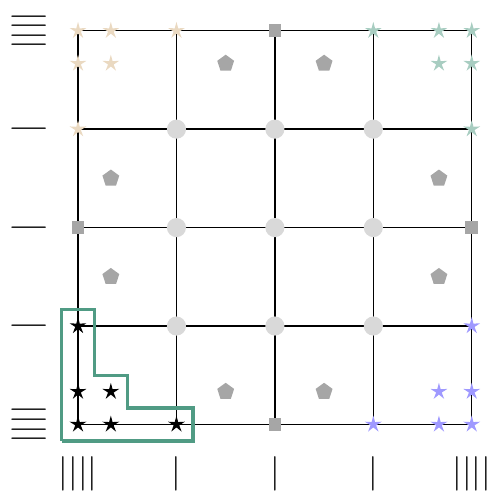}
  \subcaption{The dofs of the vertex basis functions.}
 \end{subfigure}
 \\
 \begin{subfigure}[b]{0.16\textwidth}
  \includegraphics[width=\textwidth]{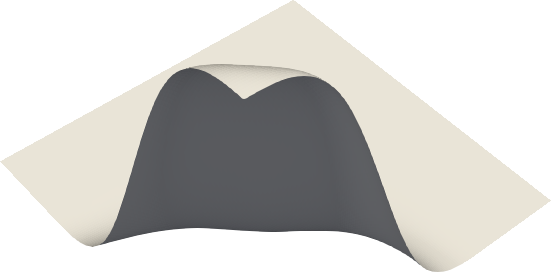}
  \subcaption{$g_s^{(k)} [\Phi_{1, 1}^{(k)}]$} \label{subfig:vertexbasis1}
 \end{subfigure}
 \begin{subfigure}[b]{0.16\textwidth}
  \includegraphics[width=\textwidth]{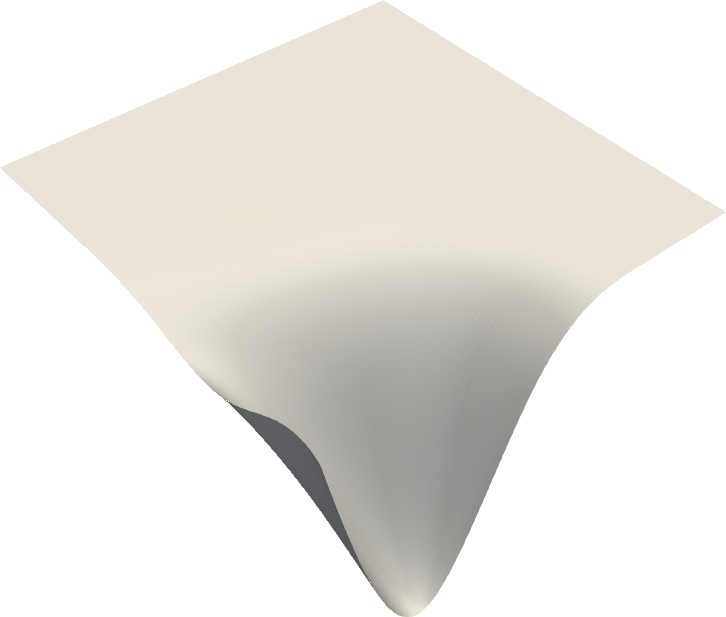}
  \subcaption{$g_s^{(k)} [\Phi_{2,1}^{(k)}]$}
 \end{subfigure}
  \begin{subfigure}[b]{0.16\textwidth}
  \includegraphics[width=\textwidth]{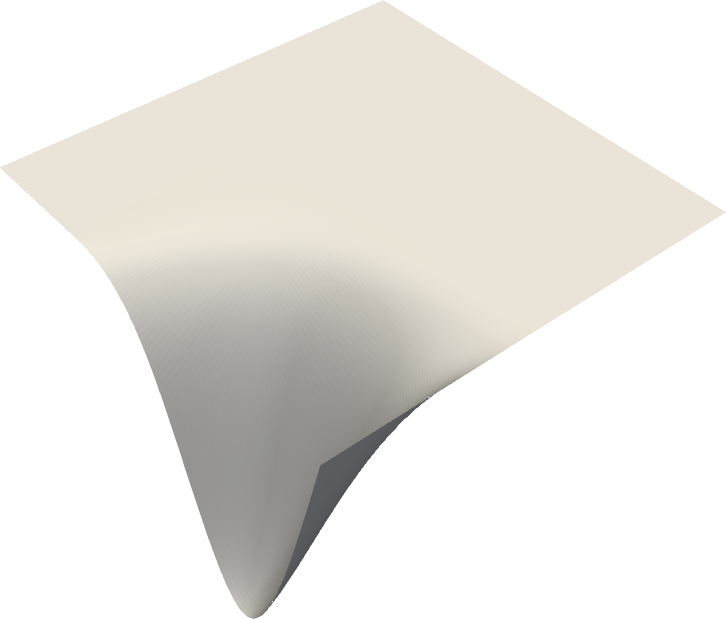}
  \subcaption{$g_s^{(k)} [\Phi_{1,2}^{(k)}]$}
 \end{subfigure}
  \begin{subfigure}[b]{0.16\textwidth}
  \includegraphics[width=\textwidth]{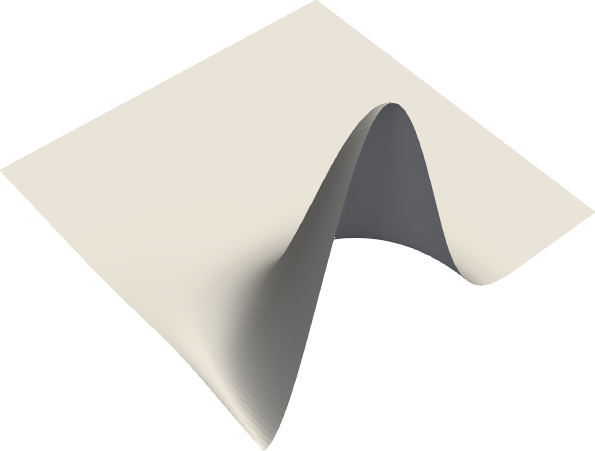}
  \subcaption{$g_s^{(k)} [\Phi_{3, 1}^{(k)}]$}
 \end{subfigure}
 \begin{subfigure}[b]{0.16\textwidth}
  \includegraphics[width=\textwidth]{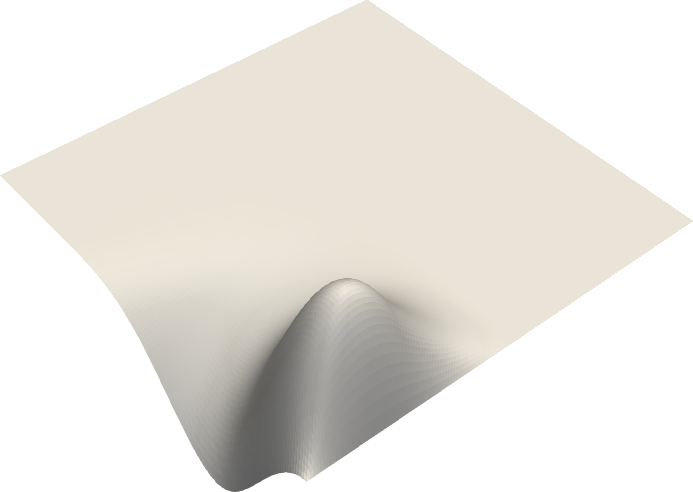}
  \subcaption{$g_s^{(k)} [\Phi_{2, 2}^{(k)}]$}
 \end{subfigure}
  \begin{subfigure}[b]{0.16\textwidth}
  \includegraphics[width=\textwidth]{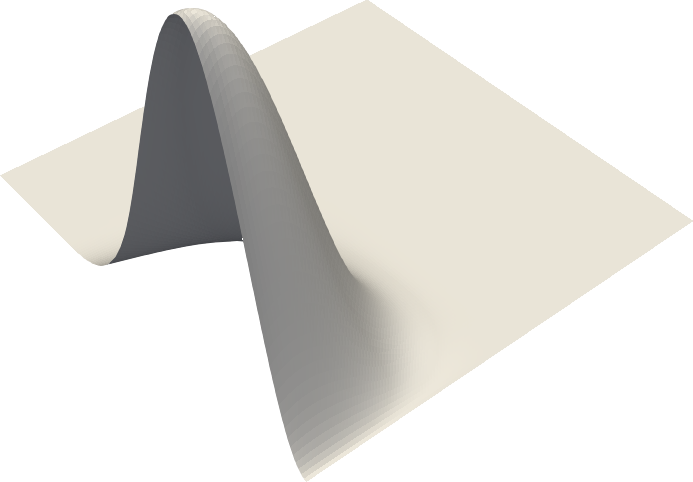}
  \subcaption{$g_s^{(k)} [\Phi_{1,3}^{(k)}]$} \label{subfig:vertexbasis6}
 \end{subfigure}
 \caption{As an example the basis functions at the vertex $u=v=0$ are shown in Subfigures~\ref{subfig:vertexbasis1}-\ref{subfig:vertexbasis6}. Those basis functions are computed by the $C^2$-interpolation at the vertex.} \label{fig:vertexspace}
\end{figure}

\subsection{Construction of the global space} \label{sec:coc1coupling}

In this subsection we first describe the coupling conditions. Then, the coupling conditions are used to connect the local (patch-wise) spaces to define the global space $\widetilde{\mathcal{V}}_{h}$. Considering one interface $I_\kappa$ between patches $\Omega^{(k)}$ and $\Omega^{(l)}$, we assume for all isogeometric functions $\varphi_h\in \widetilde{\mathcal{V}}_{h}$ that 
\begin{align}
    \varphi_h|_{E_{s_k}^{(k)}} &= \varphi_h|_{E_{s_l}^{(l)}}, \label{eq:c0condition} \\
  \partial_{\widetilde{\f n}_k} \varphi_h|_{E_{s_k}^{(k)}} &= -\partial_{\widetilde{\f n}_l} \varphi_h|_{E_{s_l}^{(l)}}. \label{eq:c1condition}
\end{align}
Moreover, we assume that for each vertex $V_\iota$ the functions $\varphi_h\in \widetilde{\mathcal{V}}_{h}$ are $C^2$-smooth at the vertex, that is, the limit of the function value, gradient and Hessian is the same on all patches sharing the vertex $V_\iota$.
\begin{rmk}
We recall the estimate from \cite[Theorem~1]{WEINMULLER2021114017}, yielding the bound
\begin{align}
  \norm{\jump{\partial_{\f n} \varphi_h}}_{L^2 (I_\kappa)} \leq C h_2^{\widetilde{p}+1} \left( \norm{\varphi_h}_{H^2(\Omega^{(k)})} + \norm{\varphi_h}_{H^2(\Omega^{(l)})} \right)^{1/2}, \label{eq:jumpnormalderivative}
\end{align}
where $\varphi_h$ is defined satisfying~\eqref{eq:c0condition}-\eqref{eq:c1condition}. Here $C>0$ depends on the geometry, but not on the mesh size. Therefore, the $C^1$ error depends on the choice of the polynomial degree for the approximate gluing data. We also observe that higher polynomial degree and/or lower regularity for the approximated gluing data does not lead to better results in the numerical experiments, see~\cite{WEINMULLER2021114017}.
\end{rmk}

Let
\[
 \widetilde{\mathcal{V}}_{h} = \left( \bigoplus_{k \in \mathcal{M}_P} \mathcal{A}_\circ^{(k)} \right) \oplus \left( \bigoplus_{\kappa \in \mathcal{M}_I} \mathcal{A}_{I_\kappa} \right) \oplus \left( \bigoplus_{\iota \in \mathcal{M}_V} \mathcal{A}_{V_\iota} \right) \oplus \left( \bigoplus_{\sigma \in \mathcal{M}_E} \mathcal{A}_{B_\sigma} \right),
\]
where the patch interior spaces are defined as
\[
 \mathcal{A}_\circ^{(k)} = \{ \varphi_h \in C^0(\Omega) : \varphi_h \circ \f F^{(k)} \in \widehat{\mathcal{A}}_\circ^{(k)} \mbox{ and }\varphi_h \circ \f F^{(l)} \equiv 0 \mbox{ for all }l\neq k \},
\]
the interface spaces as 
\[
 \mathcal{A}_{I_\kappa} = 
 \left\{ \varphi_h \in C^0(\Omega) : 
 \begin{array}{l}
  \varphi_h \circ \f F^{(k)} \in \widehat{\mathcal{A}}_{E,s_k}^{(k)}, \\
  \varphi_h \circ \f F^{(l)} \in \widehat{\mathcal{A}}_{E,s_l}^{(l)}, \\ 
  \varphi_h \circ \f F^{(m)} \equiv 0 \mbox{ for all }m\notin \{k,l\} \mbox{ and } \\
  \varphi_h \mbox{ satisfies \eqref{eq:c1condition} for }I_\kappa
 \end{array}
 \right\},
\]
the vertex spaces as 
\[
 \mathcal{A}_{V_\iota} = 
 \left\{ \varphi_h \in C^0(\Omega) : 
 \begin{array}{l}
  \varphi_h \circ \f F^{(k_i)} \in \widehat{\mathcal{A}}_{V,s_{k_i}}^{(k_i)}, \mbox{ for all }i=1,\ldots,\nu \\
  \varphi_h \circ \f F^{(l)} \equiv 0 \mbox{ for all }l\notin \{k_1,\ldots,k_\nu\} \mbox{ and } \\
  \varphi_h \mbox{ is $C^2$ at }V_\iota
 \end{array}
 \right\},
\]
and the boundary edge spaces as 
\[
 \mathcal{A}_{B_\sigma} = 
 \left\{ \varphi_h \in C^0(\Omega) : 
 \begin{array}{l}
  \varphi_h \circ \f F^{(k)} \in \widehat{\mathcal{A}}_{E,s_k}^{(k)}, \\
  \varphi_h \circ \f F^{(l)} \equiv 0 \mbox{ for }l\neq k
 \end{array}
 \right\}.
\]
A basis for the global space can be derived immediately from the local bases. For patch interior and boundary edge spaces no coupling is needed, thus the patch-local basis functions are also global basis functions. For each interface there is a direct one-to-one correspondence between basis functions on each side, that is, $f^{(k)}_{s_k} [b_j^+, 0]$ is coupled with $f^{(l)}_{s_l} [b_j^+, 0]$ and $f^{(k)}_{s_k} [0,b_j^-]$ is coupled with $f^{(l)}_{s_l} [0,b_j^-]$. The resulting basis functions are denoted with ${E}_{\kappa} [b^+, b^-]$. Similarly, the vertex basis functions $V_{\iota} [\Phi_{i_1,i_2}]$ are coupled due to the $C^2$-interpolation conditions, that is, the functions $g_{s_{k_1}}^{(k_1)} [\Phi_{i_1, i_2}^{(k_1)}]$, $g_{s_{k_2}}^{(k_2)} [\Phi_{i_1, i_2}^{(k_2)}]$, \ldots, $g_{s_{k_\nu}}^{(k_\nu)} [\Phi_{i_1, i_2}^{(k_\nu)}]$ are coupled for each index pair $(i_1,i_2) \in \mathcal{I}_V$. In Figure~\ref{fig:globalspace} an example visualizing edge and vertex basis functions is given. Since the structure of the construction is similar to the AS-$G^1$ construction in, for example,~\cite{kapl2017dimension, kapl2019argyris} we obtain linear independent basis functions for the space $\widetilde{\mathcal{V}}_{h}$ as stated in the following lemma.
\begin{lem}
 The space $\widetilde{\mathcal{V}}_{h}$ is the direct sum of its subspaces $\mathcal{A}_\circ^{(k)}$, $\mathcal{A}_{I_\kappa}$, $\mathcal{A}_{V_\iota}$ and $\mathcal{A}_{B_\sigma}$. Moreover, the coupling described above yields a basis for each of the subspaces, which in turn yields a global basis. 
\end{lem}

\begin{figure}
 \centering
  \begin{subfigure}[b]{0.32\textwidth}
    \resizebox{\textwidth}{!}{
    \includegraphics[width=\textwidth]{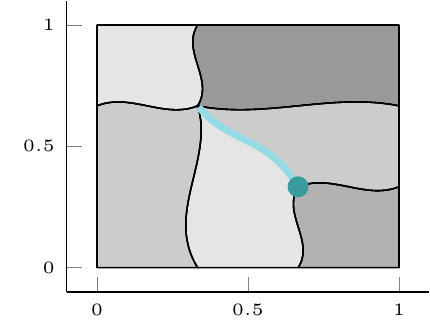}
}
  \subcaption{The chosen geometry example.}
 \end{subfigure}
 \begin{subfigure}[b]{0.22\textwidth}
  \includegraphics[width=\textwidth]{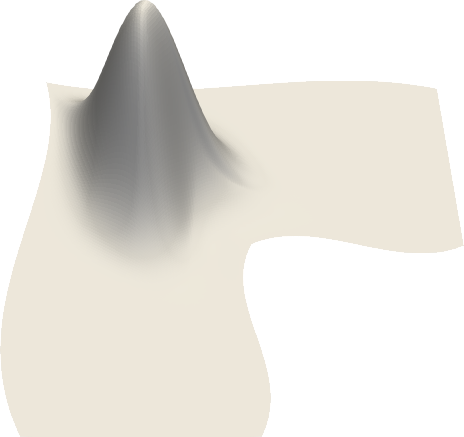}
  \subcaption{${E}_{\kappa} [b_4^+, 0]$}
 \end{subfigure}
 \begin{subfigure}[b]{0.22\textwidth}
  \includegraphics[width=\textwidth]{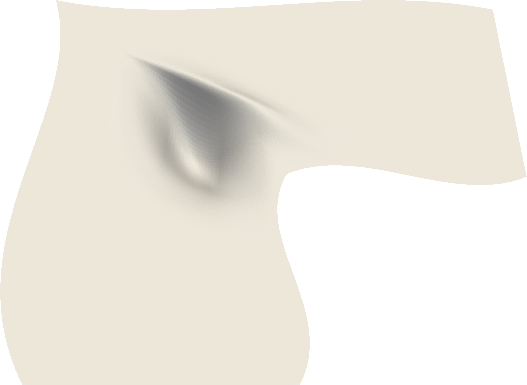}
  \subcaption{${E}_{\kappa} [0, b_3^-]$}
 \end{subfigure}
  \begin{subfigure}[b]{0.22\textwidth}
  \includegraphics[width=\textwidth]{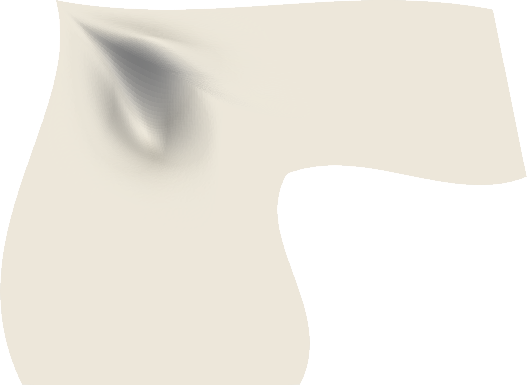}
  \subcaption{${E}_{\kappa} [0, b_4^-]$}
 \end{subfigure}
 
  \begin{subfigure}[b]{0.16\textwidth}
  \includegraphics[width=\textwidth]{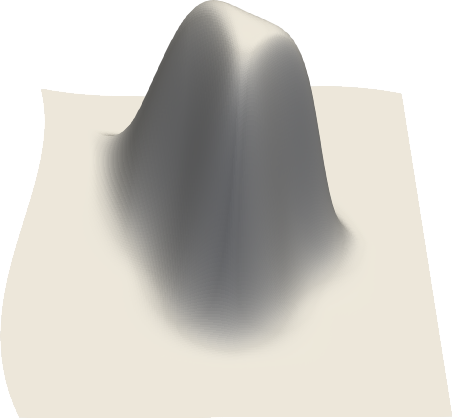}
  \subcaption{$V_{\iota} [\Phi_{1,1}]$}
 \end{subfigure}
 \begin{subfigure}[b]{0.16\textwidth}
  \includegraphics[width=\textwidth]{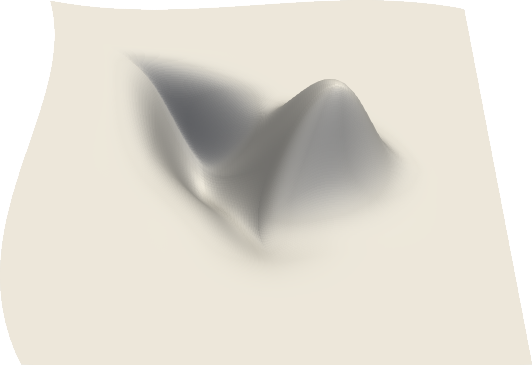}
  \subcaption{$V_{\iota} [\Phi_{2,1}]$}
 \end{subfigure}
  \begin{subfigure}[b]{0.16\textwidth}
  \includegraphics[width=\textwidth]{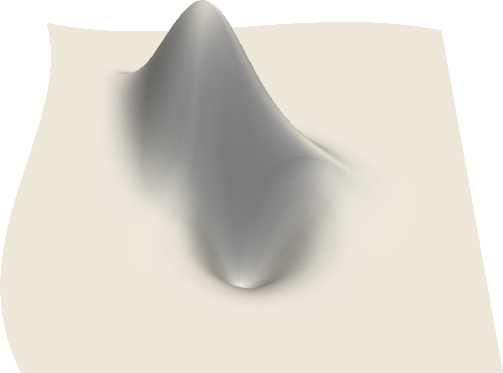}
  \subcaption{$V_{\iota} [\Phi_{1,2}]$}
 \end{subfigure}
   \begin{subfigure}[b]{0.16\textwidth}
  \includegraphics[width=\textwidth]{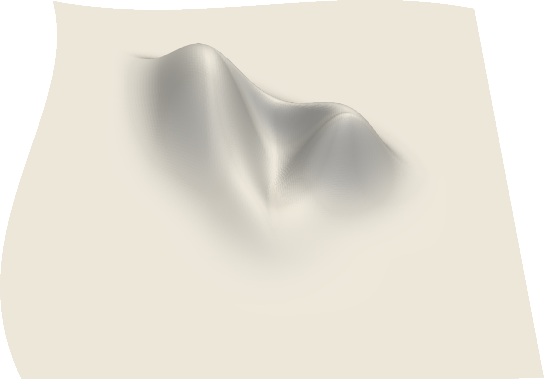}
  \subcaption{$V_{\iota} [\Phi_{2,2}]$}
 \end{subfigure}
 \begin{subfigure}[b]{0.16\textwidth}
  \includegraphics[width=\textwidth]{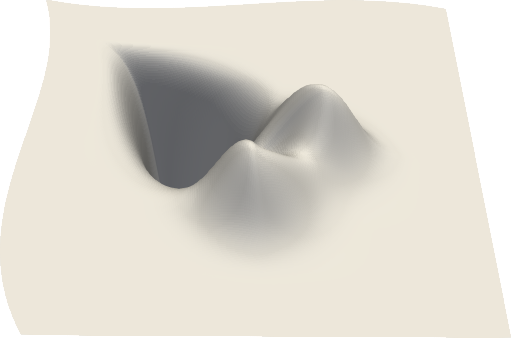}
  \subcaption{$V_{\iota} [\Phi_{3,1}]$}
 \end{subfigure}
  \begin{subfigure}[b]{0.16\textwidth}
  \includegraphics[width=\textwidth]{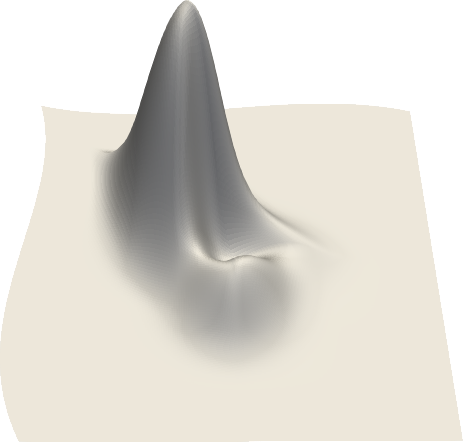}
  \subcaption{$V_{\iota} [\Phi_{1,3}]$}
 \end{subfigure}
 \caption{Examples of global basis functions at the interface (marked with a thick line) and vertex (marked with a dot). Only the relevant patches are plotted.} \label{fig:globalspace}
\end{figure}

\begin{rmk} \label{rmk:exactC1smoothness}
 We now briefly discuss the $C^1$ smoothness of the space $\widetilde{\mathcal{V}}_{h}$ which is discussed in more details in~\cite[Subsection~6.3]{WEINMULLER2021114017}. In some cases, the $C^1$ condition in~\eqref{eq:c1condition} at the interface $I_\kappa$ is exact, i.e., $\widetilde{\f n}_k = -\widetilde{\f n}_l$. A sufficient condition would be $\widetilde{\f n}_k = {\f n}_k = -{\f n}_l = -\widetilde{\f n}_l$. This is the case, when the projection of the gluing data is exact, i.e., $\widetilde{\alpha}^{(S)}(v) = \alpha^{(S)}(v)$ and $\widetilde{\beta}^{(S)}(v) = \beta^{(S)}(v)$ for $S\in \{k,l\}$. The condition holds for all $v \in [0,1]$ if the gluing data satisfies $\alpha^{(k)},\beta^{(k)}\in \mathcal{S}(\widetilde{p},\widetilde{r},h)$. Then the condition in~\eqref{eq:c1condition} is actually an exact $C^1$ condition and the jump in~\eqref{eq:jumpnormalderivative} vanishes. As a result we then have $\mathcal{A}_{I_{\kappa}} \subseteq C^1 (\Omega)$.
\end{rmk}

\subsection{Imposing inhomogeneous boundary conditions}\label{sec:imposing-bc}

Recall the boundary conditions
\begin{align}
  \left.
  \begin{array}{ll}
   \varphi &= g_0 \\
   \partial_{\f n} \varphi &= g_1
  \end{array}\right\}
   \quad &\text{ on } \Gamma_N \quad \text{ and } \\
  \left.
  \begin{array}{ll}
   \varphi &= g_0 \\
   \Delta \varphi &= g_2
  \end{array}\right\}
   \quad &\text{ on } \Gamma_L,
\end{align}
We assume that each set $\Gamma_N$ and $\Gamma_L$ is the union of boundary edges of patches, i.e., the boundary conditions can change only at vertices of the multi-patch domain.

The boundary condition $\Delta \varphi = g_2$ is naturally enforced in the equation on the right hand side. The other two boundary conditions are enforced strongly by encorporating them into the space. To do so we need to define functions that satisfy the boundary conditions for general $g_0$ and $g_1$, as well as functions that have homogeneous boundary conditions spanning the space
\[
 \widetilde{\mathcal{V}}_{h,0} = \{ \varphi_h \in \widetilde{\mathcal{V}}_{h} : \varphi = 0 \mbox{ on }\partial\Omega \mbox{ and }\partial_{\f n} \varphi = 0 \text{ on } \Gamma_N\}.
\]
We collect all functions which are used to approximate $g_0$ and $g_1$ in the space $\widetilde{\mathcal{V}}_{h,\partial \Omega} \subset \widetilde{\mathcal{V}}_{h}$. We then have
\[
 \widetilde{\mathcal{V}}_{h,\partial \Omega} \oplus \widetilde{\mathcal{V}}_{h,0} = \widetilde{\mathcal{V}}_{h}.
\]
Similar to the global space $\widetilde{\mathcal{V}}_{h}$, we split the space $\widetilde{\mathcal{V}}_{h,\partial \Omega}$ into separate contributions
\[
 \widetilde{\mathcal{V}}_{h,\partial \Omega} = \left( \bigoplus_{\iota \in \mathcal{M}_{V,\partial\Omega}} \mathcal{A}_{V_\iota,\partial \Omega} \right) \oplus \left( \bigoplus_{\sigma \in \mathcal{M}_E} \mathcal{A}_{B_\sigma,\partial \Omega} \right),
\]
where $\mathcal{M}_{V,\partial\Omega}$ denotes the indices of all boundary vertices.

Let $B_\sigma =  E^{(k)}_s$ be a boundary edge. For $B_\sigma \subset \Gamma_N$ we set $\mathcal{A}_{B_\sigma,\partial \Omega} = \mathcal{A}_{B_\sigma}$, collecting all functions from $\widehat{\mathcal{A}}_{E,s,+}^{(k)}$ and $\widehat{\mathcal{A}}_{E,s,-}^{(k)}$, whereas for $B_\sigma \subset \Gamma_L$ we define $\mathcal{A}_{B_\sigma,\partial \Omega} = \mbox{span}\{ \widehat{\mathcal{A}}_{E,s,+}^{(k)} \circ (\f F^{(k)})^{-1} \}$ to contain only those functions constructed from $\widehat{\mathcal{A}}_{E,s,+}^{(k)}$.

Let $V_\iota$ be the vertex at the boundary and the corresponding vertex space is constructed by
\[
 {\mathcal{A}}_{V_\iota} = \text{span} \{ V_{\iota} [\Phi_{1,1}], V_{\iota} [\Phi_{2,1}], V_{\iota} [\Phi_{3,1}], V_{\iota} [\Phi_{1,2}], V_{\iota} [\Phi_{1,3}], V_{\iota} [\Phi_{2,2}]  \}. 
\]
To obtain the correct subspace $\mathcal{A}_{V_\iota,\partial \Omega}$, we compute the kernel of the space, evaluated with the value, that is
\begin{align*}
 \text{ker} \langle \mathcal{A}_{V_\iota}  \rangle &= \left\{ \varphi_h = \lambda_1 V_{\iota} [\Phi_{1,1}] + \lambda_2 V_{\iota} [\Phi_{2,1}] + \lambda_3 V_{\iota} [\Phi_{3,1}] + \lambda_4 V_{\iota} [\Phi_{1,2}] + \lambda_5 V_{\iota} [\Phi_{1,3}] + \lambda_6 V_{\iota} [\Phi_{2,2}]
 \right. \\
  & \quad \left. \text{with } \lambda_i \in \mathbb{R}, \, i \in \{1,...,6\} \; : \; \varphi_h|_{\partial \Omega} = 0 \text{ and } \partial_{\f n} \varphi_h |_{\Gamma_N} = 0 \right\}.
\end{align*}
For the boundary space, we use the functions which span
\[
 \mathcal{A}_{V_\iota,\partial \Omega} = \mathcal{A}_{V_\iota} \setminus \text{ker} \langle \mathcal{A}_{V_\iota} \rangle.
\]
Note that the dimension of the kernel and of the boundary space depends on the considered boundary conditions and on the geometry.
\begin{rmk}
To obtain an exact kernel at the vertices, an interpolation of the approximate gluing data at the boundary points is required. However, the interpolation can be omitted. In this case, the kernel must be calculated to a $h$-dependent tolerance.
\end{rmk}

\section{Existence and uniqueness of the solution using Nitsche's method} \label{sec:discreteSpaceNitsche}

In this section, the optimal choice of the stability parameter yielding coercivity and boundedness of the bilinear form of Problem~\ref{problem:discreteformulationnitsche} is derived. Thus, the existence and uniqueness of the solution of Problem~\ref{problem:discreteformulationnitsche} is shown. We prove the coercivity and boundedness of the form in the following dG-norm
\begin{align*}
  \norm{\varphi}_{\mathcal{X}_h}^2 = (\varphi, \varphi)_{\mathcal{X}_h} \quad \text{where} \quad (\varphi, \psi)_{\mathcal{X}_h} \coloneqq (\Delta\varphi, \Delta\psi)_{\mathcal{H}^0 (\Omega)} + (\varphi, \psi)_{C_h}.
\end{align*}
We start with a description of the discrete space for Nitsche's method. Let the space $\mathcal{X}_{h}^{(k)}$ be the standard tensor-product spline space for the patch $k$. Then the space $\mathcal{X}_{h}$ is the collection of the patch spaces $\mathcal{X}_{h}^{(k)}$ and, in addition, the dofs along the interfaces are matching to ensure $C^0$ smoothness, see,  e.g.,~\cite{cottrell2009isogeometric, scott2014isogeometric}. For the Problem~\ref{problem:discreteformulationnitsche}, we set
\[
 \mathcal{X}_{h,0} = \mathcal{X}_h \cap \mathcal{X}_0
\]
to fulfill the boundary conditions. In order to prove the coercivity and boundedness, we need to bound the average at the interface which is bounded as follows:
\begin{lem} \label{lem:averageboundness}
  Let $\varphi_h \in \mathcal{X}_h$. For any $I_{\kappa}$ with $\kappa \in \mathcal{M}_I$, we have
 \[
  \norm{\average{\Delta \varphi_h}_{\kappa}}_{L^2(I_{\kappa})}^2 \leq c_\kappa (h) \left( \norm{\Delta \varphi_h}_{L^2 (\Omega^{(k)})}^2 + \norm{\Delta \varphi_h}_{L^2 (\Omega^{(l)})}^2 \right).
 \]
 where $c_{\kappa} (h)>0$ is a constant, which depends on the mesh size $h$ and on the patch parametrizations, but not on the function $\varphi_h$.
\end{lem}
\begin{proof}
 Using the definition of the average, we have with the triangle inequality
 \[
  \norm{\average{\Delta \varphi_h}_{\kappa}}_{L^2(I_{\kappa})}^2 \leq \frac12 \left( \norm{\Delta \varphi_h}_{L^2(E_{s_k}^{(k)})}^2 + \norm{\Delta \varphi_h}_{L^2(E_{s_l}^{(l)})}^2 \right).
 \]
 Since $\varphi_h$ is from a finite dimensional space, the supremum
 \[
  \sup_{\varphi_h \in \mathcal{X}_h, \norm{\Delta \varphi_h}_{L^2(\Omega^{(k)})}=1} \norm{\Delta \varphi_h}_{L^2(E_{s_k}^{(k)})}^2
 \]
 exists and we denote it by $c_k$. Similarly, we obtain an upper bound $c_l$ on $\Omega^{(l)}$. Thus, the proof is complete with $c_\kappa = \frac{1}{2} (c_k+c_l)$, which does, in general, depend on $h$.
\end{proof}
\begin{ass} \label{ass:meshindependent}
 For any $I_{\kappa}$ with $\kappa \in \mathcal{M}_I$, the constant $c_\kappa (h)$ from Lemma~\ref{lem:averageboundness} satisfies 
 \[
  c_\kappa (h) \leq \frac{\overline{c}_\kappa}{h},
 \]
 where $\overline{c}_\kappa$ is an $h$-independent constant.
\end{ass}
In practice, the constant $c_\kappa(h)$ in Lemma~\ref{lem:averageboundness} can be computed by a generalized eigenvalue problem, following the same steps as in~\cite{embar2010imposing}, for example. For $\eta_{\kappa}$ sufficiently large, Nitsche's formulation as given in Problem~\ref{problem:discreteformulationnitsche} is coercive and bounded which is stated in the following Theorem.
\begin{thm} \label{thm:coerciveboundedness}
  For all $\varphi_h, \psi_h \in \mathcal{X}_{h,0}$, let $\eta_\kappa$ be such that $\eta_\kappa > 16 \, h \, c_\kappa(h)$, then
  \begin{align*}
    (\varphi_h,\varphi_h)_{A_h} \geq \underline{\mu} \norm{\varphi_h}_{\mathcal{X}_h}^2 \quad \text{and} \quad (\varphi_h,\psi_h)_{A_h} \leq \overline{\mu} \norm{\varphi_h}_{\mathcal{X}_h} \norm{\psi_h}_{\mathcal{X}_h}
  \end{align*}
  where $c_\kappa (h)$ is the constant from Lemma~\ref{lem:averageboundness} and $\underline{\mu}, \overline{\mu}$ are constants depending on $c_\kappa (h)$. Considering Assumption~\ref{ass:meshindependent} we require $\eta_\kappa > 16 \, \overline{c}_\kappa$ and thus $\underline{\mu}, \overline{\mu}$ are constants independent of the mesh-size.
\end{thm}
\begin{proof}[Proof of Theorem~\ref{thm:coerciveboundedness}]
 By definition, we have
\begin{align*}
  (\varphi_h,\varphi_h)_{A_h}
  &=  (\varphi, \psi)_{\mathcal{X}_h} - (\varphi, \psi)_{B_h} - (\psi, \varphi)_{B_h} .
\end{align*}
With the help of Young's inequality with $\delta_{\kappa} > 0$ and Lemma~\ref{lem:averageboundness} we have
\begin{align*}
  |(\varphi_h, \varphi_h)_{B_h}| 
   &\leq
  \sum_{\kappa \in \mathcal{M}_I} \frac{1}{\delta_{\kappa}} \norm{\jump{\partial_{\f n_\kappa} \varphi_h}}_{L^2 (I_{\kappa})}^2  + \sum_{\kappa \in \mathcal{M}_I} \delta_{\kappa} \norm{ \average{\Delta \varphi_h} }_{L^2 (I_{\kappa})}^2  \\
   &\leq 
  \sum_{\kappa \in \mathcal{M}_I} \frac{1}{\delta_{\kappa}} \norm{\jump{\partial_{\f n_\kappa} \varphi_h}}_{L^2 (I_{\kappa})}^2  + \sum_{\kappa \in \mathcal{M}_I}  c_\kappa (h)\delta_{\kappa} ( \norm{\Delta \varphi_h}_{L^2 (\Omega^{(k)})}^2 + \norm{\Delta \varphi_h}_{L^2 (\Omega^{(l)})}^2) .
\end{align*}
Since
\[
 \norm{\Delta\varphi_h}_{\mathcal{H}^0(\Omega)}^2 \geq \frac{1}{4} \sum_{\kappa \in \mathcal{M}_I} ( \norm{\Delta \varphi_h}_{L^2 (\Omega^{(k)})}^2 + \norm{\Delta \varphi_h}_{L^2 (\Omega^{(l)})}^2)
\]
we have for the coercivity
\begin{align*}
 (\varphi_h,\varphi_h)_{A_h}
  &\geq \norm{\Delta\varphi_h}_{\mathcal{H}^0(\Omega)}^2  - \sum_{\kappa \in \mathcal{M}_I} \frac{2}{\delta_{\kappa}} \norm{\jump{\partial_{\f n_\kappa} \varphi_h}}_{L^2 (I_{\kappa})}^2  - \sum_{\kappa \in \mathcal{M}_I}  2 c_\kappa (h)\delta_{\kappa} ( \norm{\Delta \varphi_h}_{L^2 (\Omega^{(k)})}^2 + \norm{\Delta \varphi_h}_{L^2 (\Omega^{(l)})}^2) \\
  & \qquad + \sum_{\kappa \in \mathcal{M}_I} \frac{\eta_\kappa}{h} \norm{\jump{\partial_{\f n_\kappa} \varphi_h}}_{L^2 (I_{\kappa})}^2 \\
  &\geq \sum_{\kappa \in \mathcal{M}_I} (\frac{1}{4}- 2c_\kappa (h)\delta_{\kappa}) ( \norm{\Delta \varphi_h}_{L^2 (\Omega^{(k)})}^2 + \norm{\Delta \varphi_h}_{L^2 (\Omega^{(l)})}^2) + \sum_{\kappa \in \mathcal{M}_I} (1- \frac{2h}{\delta_{\kappa} \eta_\kappa}) \frac{\eta_\kappa}{h}\norm{\jump{\partial_{\f n_\kappa} \varphi_h}}_{L^2 (I_{\kappa})}^2 ,  
\end{align*}
thus the bilinear form is coercive, if
\[
 \frac{1}{4}- 2c_\kappa (h)\delta_{\kappa} > 0 \quad \Rightarrow \quad \frac{1}{\delta_{\kappa}} > 8c_\kappa(h)
\]
and 
\[
 1- \frac{2h}{\delta_{\kappa} \eta_\kappa} > 0 \quad \Rightarrow \quad  \eta_\kappa > \frac{2h}{\delta_{\kappa}} > 16 \, h \, c_\kappa(h)
\]
for all $\kappa$. This means that the stability parameter $\eta_\kappa$ must satisfy $\eta_\kappa > 16 \, h \, c_\kappa(h)$ and, under Assumption~\ref{ass:meshindependent}, that $\eta_\kappa > 16 \, \overline{c}_\kappa$.

For the boundedness we obtain from the Cauchy-Schwarz inequality, the triangle inequality, Lemma~\ref{lem:averageboundness} and the fact that $s \in \{1,...,4\}$
\begin{align*}
  |(\varphi_h, \psi_h)_{B_h}| &\leq
  \left(\sum_{\kappa \in \mathcal{M}_I} \norm{\jump{\partial_{\f n_\kappa} \psi_h}}_{L^2 (I_{\kappa})}^2 \right)^{1/2} \left(\sum_{\kappa \in \mathcal{M}_I} \norm{ \average{\Delta \varphi_h} }_{L^2 (I_{\kappa})}^2 \right)^{1/2} \\
  &\leq \left(\sum_{\kappa \in \mathcal{M}_I} \norm{\jump{\partial_{\f n_\kappa} \psi_h}}_{L^2 (I_{\kappa})}^2 \right)^{1/2} \left( c_\kappa(h)  \sum_{\kappa \in \mathcal{M}_I} \left( \norm{\Delta \varphi_h}_{L^2 (\Omega^{(k)})}^2 + \norm{\Delta \varphi_h}_{L^2 (\Omega^{(l)})}^2 \right) \right)^{1/2} \\
  &\leq \left( \max_\kappa \frac{h \, c_\kappa(h)}{\eta_{\kappa}} \sum_{\kappa \in \mathcal{M}_I} \frac{\eta_{\kappa}}{h}\norm{\jump{\partial_{\f n_\kappa} \psi_h}}_{L^2 (I_{\kappa})}^2 \right)^{1/2} \left( \sum_{k \in \mathcal{M}_P} 4 \norm{\Delta \varphi_h }_{L^2 (\Omega^{(k)})}^2 \right)^{1/2} \\
  &\leq \max_\kappa \left( \frac{4 \, h \, c_\kappa(h)}{\eta_{\kappa}}\right)^{1/2} \norm{\psi_h}_{\mathcal{X}_h} \norm{\varphi_h}_{\mathcal{X}_h}
\end{align*}
for all $\varphi_h, \psi_h \in \mathcal{X}_{h,0}$. Using the estimate above and the Cauchy-Schwarz inequality, we obtain boundedness 
of the bilinear form
\begin{align*}
  (\varphi_h,\varphi_h)_{A_h}
  &=  (\varphi, \psi)_{\mathcal{X}_h} - (\varphi, \psi)_{B_h} - (\psi, \varphi)_{B_h} 
  \leq \left(1+2\max_\kappa \left( \frac{4 \, h \, c_\kappa(h)}{\eta_{\kappa}}\right)^{1/2} \right) \norm{\psi_h}_{\mathcal{X}_h} \norm{\varphi_h}_{\mathcal{X}_h}.
\end{align*}
If moreover Assumption~\ref{ass:meshindependent} is satisfied, then the constant is independent of $h$.
\end{proof}
Using Theorem~\ref{thm:coerciveboundedness}, we can apply the Lax-Milgram theorem, e.g., as in~\cite{evans1998partial}, and consequently have existence and uniqueness of the solution to Problem~\ref{problem:discreteformulationnitsche}. The existence and uniqueness of the solution of Problem~\ref{problem:approxC1discreteformulation} is straightforward.

\section{Numerical experiments}\label{sec:numerical-experiments}

In this section we perform numerical experiments on four $C^0$ multi-patch geometries -- denoted by Example I-IV. On a biharmonic model problem we compare the two methods considered in this work, i.e., the approximate $C^1$ discretization and Nitsche's method. More precisely, we solve on each geometry Problem~\ref{problem:approxC1discreteformulation} using the discrete space $\widetilde{\mathcal{V}}_{h,0}^1$ and  Problem~\ref{problem:discreteformulationnitsche} using the discrete space $\mathcal{X}_{h,0}$.
For simplicity, let $p = p_1^{(k)} = p_2^{(k)}$ and $r = r_1^{(k)} = r_2^{(k)}$. We consider the exact solution
\[
 \varphi(x,y) = (\cos (4\pi x)-1)(\cos (4\pi y)-1).
\]
As boundary conditions in~\eqref{eq:boundary2}-\eqref{eq:boundary3} we consider $\Gamma_L=\emptyset$ for Example~I and~II, and $\Gamma_N=\emptyset$ for Example~III and~IV.

Let $\varphi_h$ be the discrete solution of either Problem~\ref{problem:approxC1discreteformulation} or Problem~\ref{problem:discreteformulationnitsche}. Since the exact solution is smooth, we expect in both cases optimal convergence rates in the mesh size~$h$, i.e., 
\begin{align}
  \norm{\varphi - \varphi_h}_{\mathcal{H}^2 (\Omega)} = O(h^{p-1}). \label{eq:errorestimate}
\end{align}

In Subsection~\ref{sec:geometries}, the four geometries are presented. A comparison of the errors for different polynomial degrees is given in Subsection~\ref{sec:error}. In Subsection~\ref{sec:stabilizationparameter} we study the influence of the solution using Nitsche's method on the choice of the stability parameter. Finally, in Subsection~\ref{sec:reparam-v-approx} we conclude with a comparison of the (exact) $C^1$-smooth discretization on the reparametrized AS-$G^1$ geometry, following the approach presented in~\cite{kapl2018construction}, with the approximate~$C^1$ method. All tests are implemented within the open-source C\texttt{++} library G+Smo, cf. \cite{gismoweb}.

\subsection{The model geometries} \label{sec:geometries}

We consider four geometries, the first two describe the same domain, that is, the unit square, but with different parametrizations, the last two are more application-oriented geometries. The four geometries are shown in Subfigures~\ref{subfig:geometry1}--\ref{subfig:geometry4}. Example~I consists of six bilinear patches and the gluing data is consequently linear. It follows that we have exact $C^1$ smoothness at the interfaces, see Remark~\ref{rmk:exactC1smoothness}. In contrast to Example~I, Example~II is made of bicubic patches and has curved interfaces. The gluing data of this geometry is not linear and therefore, the discete space is only approximate $C^1$. The same is true for Example~III and Example~IV: both geometries have non-linear gluing data and hence the spaces are not exact $C^1$. Example~III describes a turtle with bicubic patches, while Example~IV is a part of a picture of a car and is inspired by examples from~\cite{buchegger2016adaptively,kapl2018construction}. All patches in Examples~I--IV are B\'ezier patches. The corresponding exact solutions are depicted in Figures~\ref{subfig:exactSolution1}-\ref{subfig:exactSolution4}.

\begin{figure}[h!]
\centering
%%
%% 1 Reihe
%%
\begin{subfigure}[b]{0.23\textwidth}
  \centering
  
  \resizebox{\textwidth}{!}{
  \includegraphics[width=\textwidth]{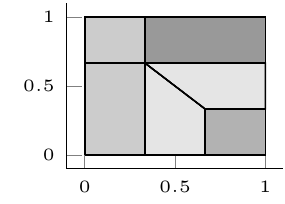}
}
  \subcaption{Ex.\ I: the multi-patch}  \label{subfig:geometry1}
\end{subfigure}
~
\begin{subfigure}[b]{0.23\textwidth}
  \centering
  
  \resizebox{\textwidth}{!}{
  \includegraphics[width=\textwidth]{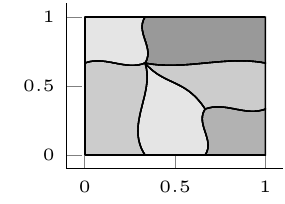}
}

  \subcaption{Ex.\ II: the multi-patch}  \label{subfig:geometry2}
\end{subfigure}
~
\begin{subfigure}[b]{0.23\textwidth}
  \centering
  
  \resizebox{\textwidth}{!}{
  \includegraphics[width=\textwidth]{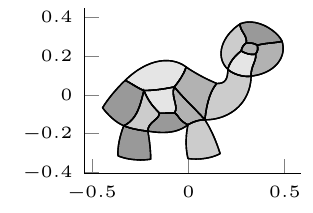}
}

  \subcaption{Ex.\ III: the multi-patch}  \label{subfig:geometry3}
\end{subfigure}
~
\begin{subfigure}[b]{0.23\textwidth}
  \centering
  
  \resizebox{\textwidth}{!}{
  \includegraphics[width=\textwidth]{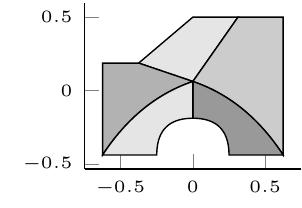}
}

  \subcaption{Ex.\ IV: the multi-patch}  \label{subfig:geometry4}
\end{subfigure}

%%
%% 2 Reihe
%%
\begin{subfigure}[b]{0.23\textwidth}
  \centering
  
  \resizebox{\textwidth}{!}{
  \includegraphics[width=\textwidth]{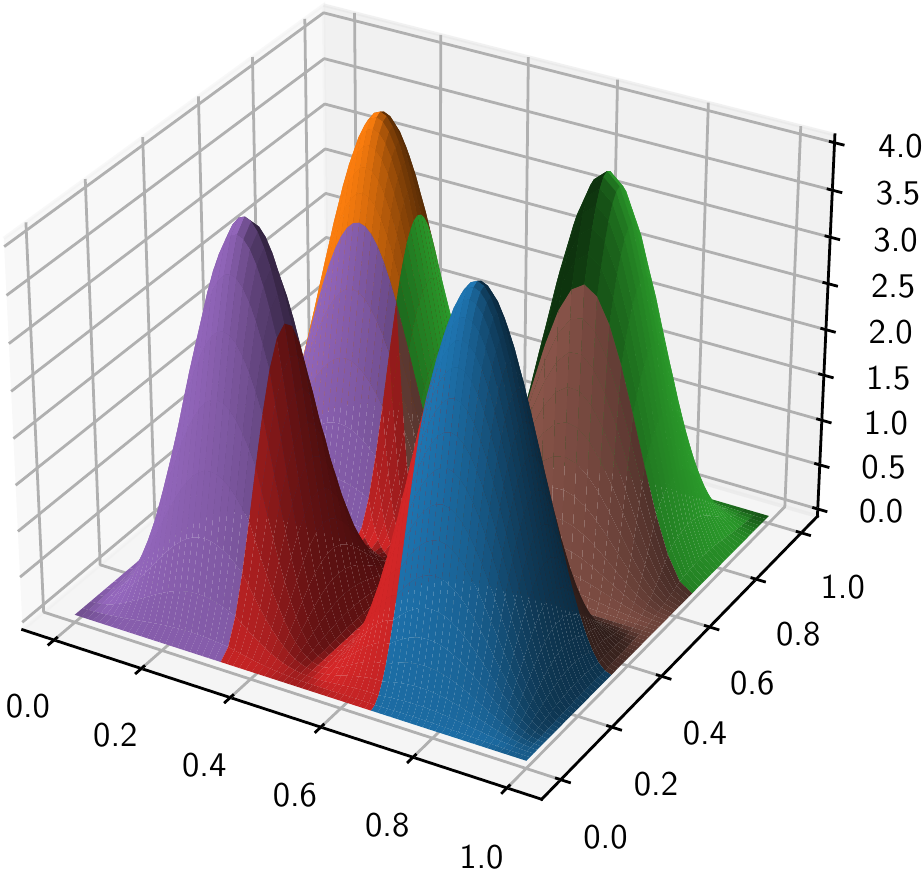}
}
  \subcaption{Ex.\ I: exact solution}  \label{subfig:exactSolution1}
\end{subfigure}
~
\begin{subfigure}[b]{0.23\textwidth}
  \centering
  
  \resizebox{\textwidth}{!}{
  \includegraphics[width=\textwidth]{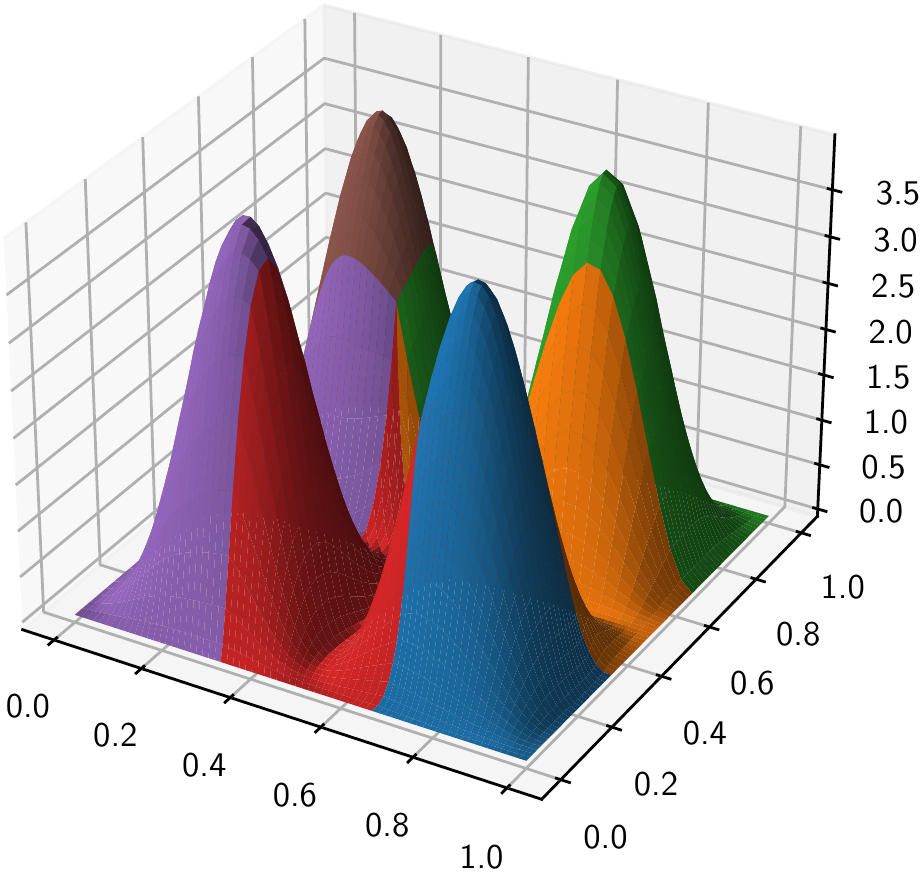}
}

  \subcaption{Ex.\ II: exact solution} 
\end{subfigure}
~
\begin{subfigure}[b]{0.23\textwidth}
  \centering
  
  \resizebox{\textwidth}{!}{
  \includegraphics[width=\textwidth]{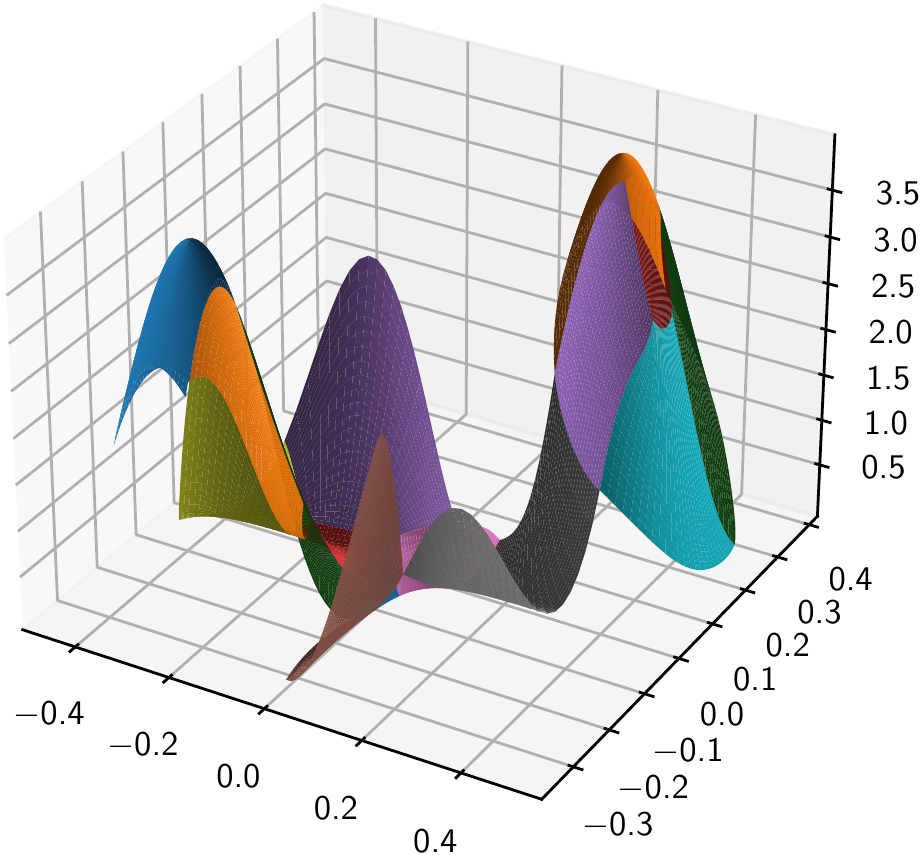}
}

  \subcaption{Ex.\ III: exact solution} 
\end{subfigure}
~
\begin{subfigure}[b]{0.23\textwidth}
  \centering
  
  \resizebox{\textwidth}{!}{
  \includegraphics[width=\textwidth]{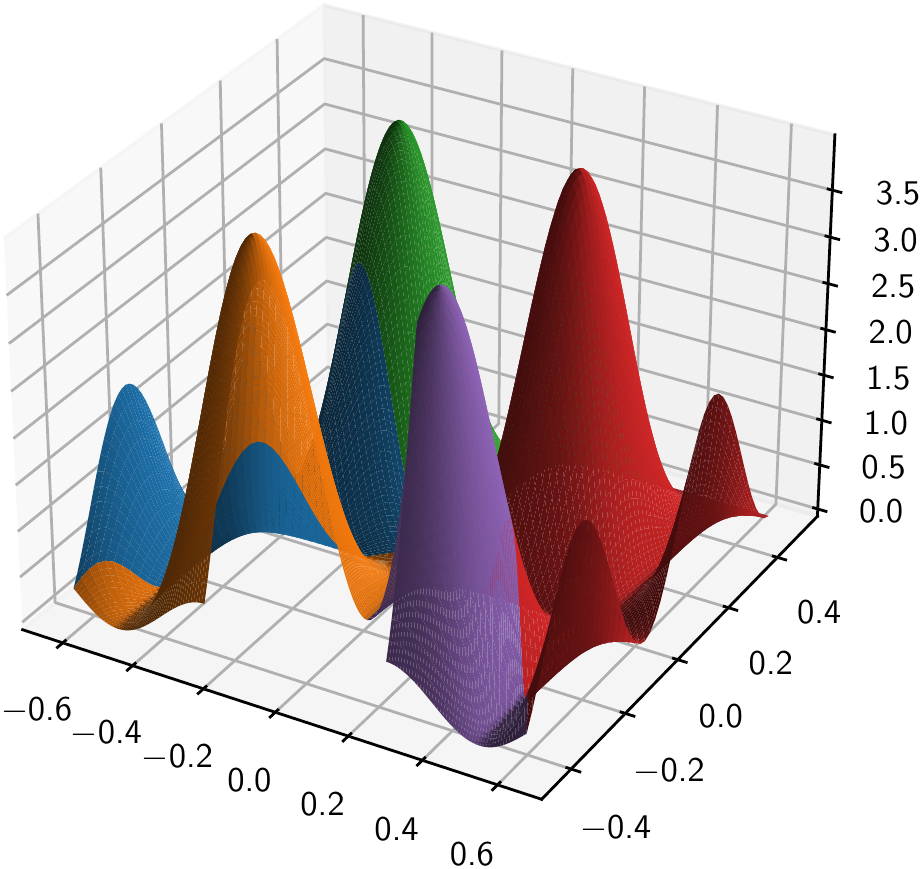}
}

  \subcaption{Ex.\ IV: exact solution}  \label{subfig:exactSolution4}
\end{subfigure}

\caption{The geometries which are used for the numerical results and their exact solutions.} \label{fig:geometry_figure}
\end{figure}

\FloatBarrier
\subsection{Convergence analysis} \label{sec:error}

We compare the convergence rates of the errors measured in the $L^2$-, $H^1$- and $H^2$-norms for the polynomial degrees $p \in \{3,4,5\}$ and use the maximum regularity $r = p - 1$. In all four examples we compute the error using the approximate $C^1$ method and Nitsche's method, represented by a dashed and solid line, respectively. To obatin the stability term in Nitsche's method, we solve Lemma~\ref{lem:averageboundness} with the eigenvalue problem at a fixed $h_0$ and choose
\[
 \eta_\kappa = \frac{4}{h_0} c_\kappa
\]
where $c_\kappa$ is the largest eigenvalue.

In Figure~\ref{fig:error_figure}, we plot the results for both methods and see that the $H^1$- and $H^2$-errors differ only slightly. In the plots, these lines almost overlap and converge optimally with the rate stated in~\eqref{eq:errorestimate}. In the $L^2$-norm, the error for the approximate $C^1$ method in Examples~I and~II is almost the same as the error for Nitsche's method, while in Examples~III and~IV the error for Nitsche's method is slightly smaller. One reason could be, that the approximate $C^1$ basis is slightly more restrictive near boundary vertices. Nevertheless, we see in the examples that both methods solve the biharmonic equation optimally with similar error values.

For Example~II, we additionally compare the errors of the approximate $C^1$ method and Nitsche's method with a single patch parametrization on the same domain (unit square). There we plot the results as a function of the number of dofs, see Figure~\ref{fig:errorvsdofs}. As expected, the single patch parametrization needs the least number of dofs for a given error. The approximated method and Nitsche's method are almost the same: the approximate $C^1$ method needs a slightly smaller number of degrees of freedom (dofs) than Nitsche's method for a fixed mesh-size. However, it can be seen that the gap between the single patch and the other two methods, with smaller mesh-size, becomes smaller. 

\begin{figure}[h!]
\begin{subfigure}[b]{0.32\textwidth}
  \centering
  
   \resizebox{\textwidth}{!}{
\includegraphics[width=\textwidth]{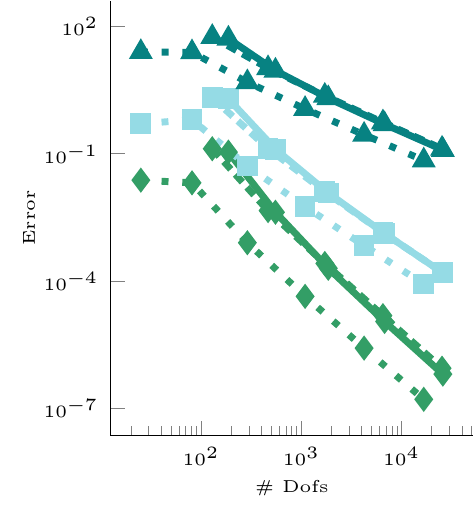}
    }
      
\subcaption{Ex.\ II, $p = 3, r = 2$} 
\end{subfigure}
\begin{subfigure}[b]{0.32\textwidth}
  \centering
  
   \resizebox{\textwidth}{!}{
\includegraphics[width=\textwidth]{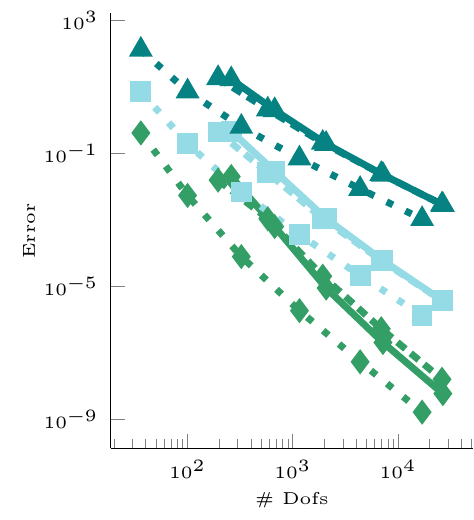}
    }
      
\subcaption{Ex.\ II, $p = 4, r = 3$} 
\end{subfigure}
\begin{subfigure}[b]{0.32\textwidth}
  \centering
  
   \resizebox{\textwidth}{!}{
\includegraphics[width=\textwidth]{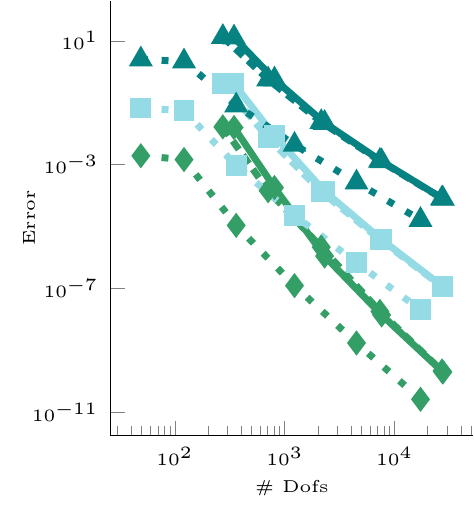}
    }
      
\subcaption{Ex.\ II, $p = 5, r = 4$} 
\end{subfigure}

\begin{center}
\begin{subfigure}[b]{0.4\textwidth}
  \centering
  
   \resizebox{\textwidth}{!}{
\includegraphics[width=\textwidth]{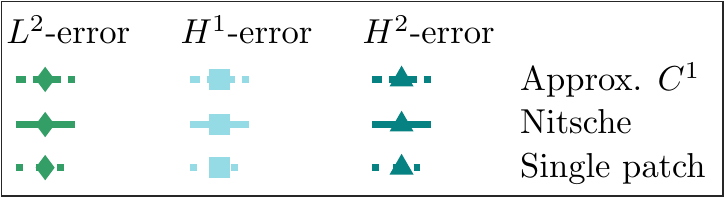}
    }
\end{subfigure}
\end{center}

\caption{The errors versus the dofs. } \label{fig:errorvsdofs}
\end{figure}

\begin{figure}[h!]

\begin{subfigure}[b]{0.23\textwidth}
  \centering
  
   \resizebox{\textwidth}{!}{
\includegraphics[width=\textwidth]{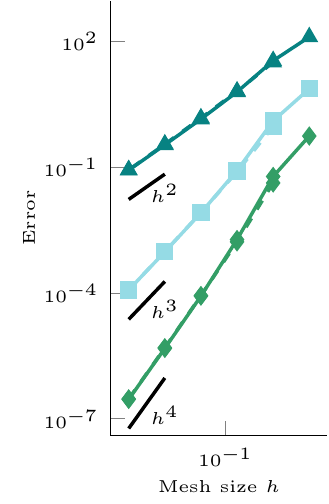}
    }
      
\subcaption{Ex.\ I, $p = 3, r = 2$} 
\end{subfigure}
\begin{subfigure}[b]{0.23\textwidth}
  \centering
  
   \resizebox{\textwidth}{!}{
\includegraphics[width=\textwidth]{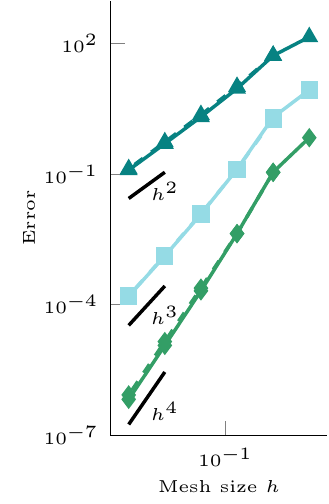}
    }
    
\subcaption{Ex.\ II, $p = 3, r = 2$}
\end{subfigure}
\begin{subfigure}[b]{0.23\textwidth}
  \centering
  
   \resizebox{\textwidth}{!}{
\includegraphics[width=\textwidth]{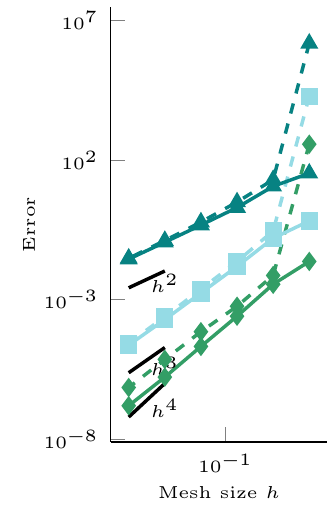}
    }
    
\subcaption{Ex.\ III, $p = 3, r = 2$}
\end{subfigure}
\begin{subfigure}[b]{0.23\textwidth}
  \centering
  
   \resizebox{\textwidth}{!}{
\includegraphics[width=\textwidth]{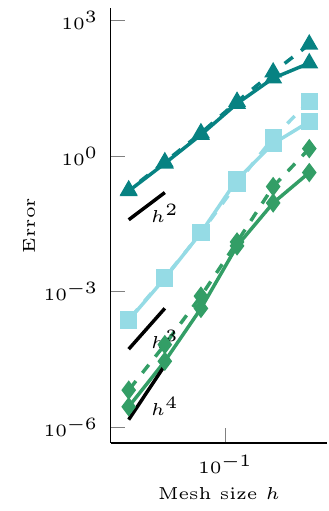}
    }
    
\subcaption{Ex.\ IV, $p = 3, r = 2$}
\end{subfigure}

%%
%% 2 row
%%
\begin{subfigure}[b]{0.23\textwidth}
  \centering
  
   \resizebox{\textwidth}{!}{
\includegraphics[width=\textwidth]{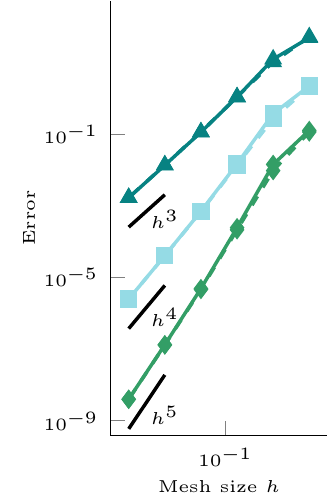}
    }
      
\subcaption{Ex.\ I, $p = 4, r = 3$} 
\end{subfigure}
\begin{subfigure}[b]{0.23\textwidth}
  \centering
  
   \resizebox{\textwidth}{!}{
\includegraphics[width=\textwidth]{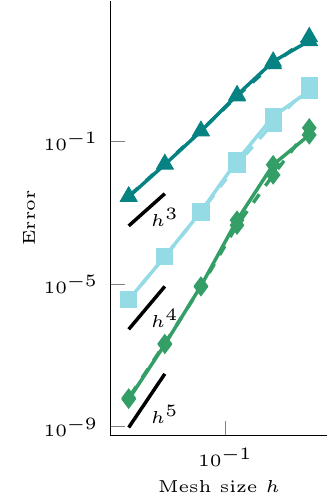}
    }
    
\subcaption{Ex.\ II, $p = 4, r = 3$}
\end{subfigure}
\begin{subfigure}[b]{0.23\textwidth}
  \centering
  
   \resizebox{\textwidth}{!}{
\includegraphics[width=\textwidth]{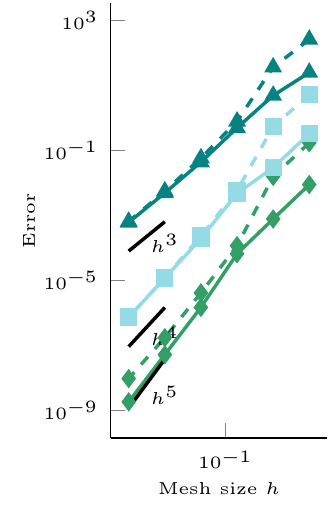}
    }
    
\subcaption{Ex.\ III, $p = 4, r = 3$}
\end{subfigure}
\begin{subfigure}[b]{0.23\textwidth}
  \centering
  
   \resizebox{\textwidth}{!}{
\includegraphics[width=\textwidth]{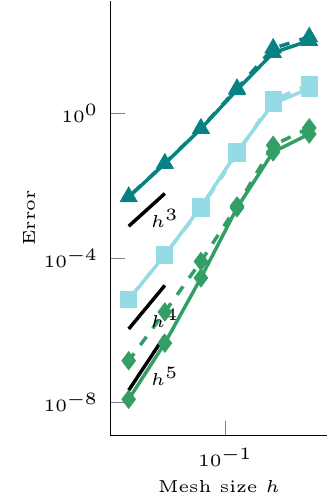}
    }
    
\subcaption{Ex.\ IV, $p = 4, r = 3$}
\end{subfigure}

%%
%% 3 row
%%
\begin{subfigure}[b]{0.23\textwidth}
  \centering
  
   \resizebox{\textwidth}{!}{
\includegraphics[width=\textwidth]{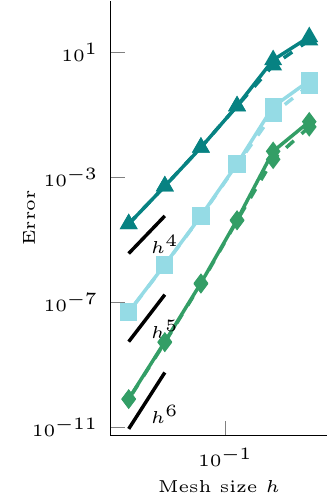}
    }
      
\subcaption{Ex.\ I, $p = 5, r = 4$} 
\end{subfigure}
\begin{subfigure}[b]{0.23\textwidth}
  \centering
  
   \resizebox{\textwidth}{!}{
\includegraphics[width=\textwidth]{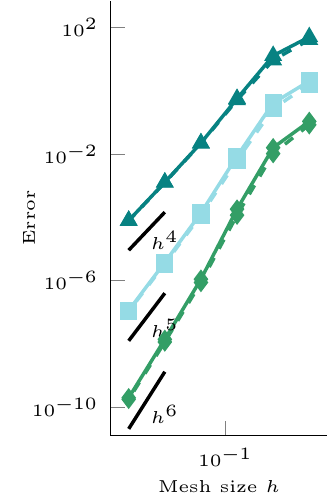}
    }
    
\subcaption{Ex.\ II, $p = 5, r = 4$}
\end{subfigure}
\begin{subfigure}[b]{0.23\textwidth}
  \centering
  
   \resizebox{\textwidth}{!}{
\includegraphics[width=\textwidth]{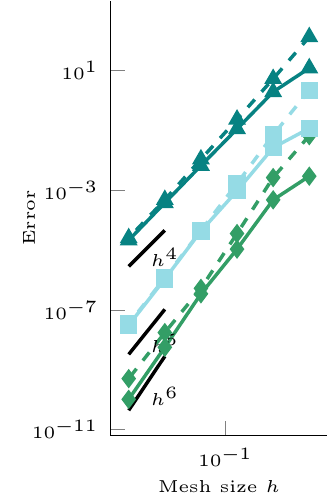}
    }
    
\subcaption{Ex.\ III, $p = 5, r = 4$}
\end{subfigure}
\begin{subfigure}[b]{0.23\textwidth}
  \centering
  
   \resizebox{\textwidth}{!}{
\includegraphics[width=\textwidth]{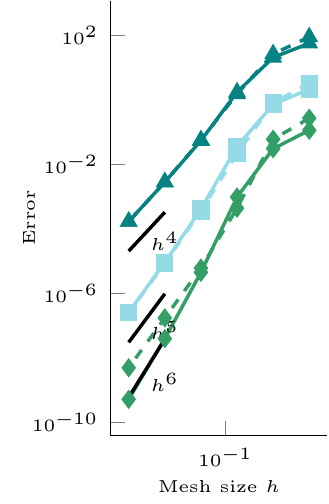}
    }
    
\subcaption{Ex.\ IV, $p = 5, r = 4$}
\end{subfigure}

\begin{center}
  \begin{subfigure}[t]{0.5\textwidth}
  \centering
  
   \resizebox{\textwidth}{!}{
   \includegraphics[width=\textwidth]{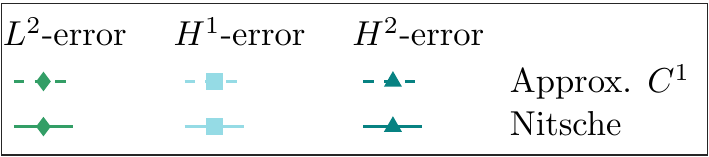}
    }

\end{subfigure}
\end{center}

\caption{Convergence rate with different polynomial degrees.} \label{fig:error_figure}
\end{figure} 

\subsection{Dependence on the stability parameter in Nitsche's method} \label{sec:stabilizationparameter}

In this subsection we examine the dependence of the error using Nitsche's method on the stability parameter. We assume that the parameter is chosen globally, i.e., the parameter is the same for each interface. Therefore, the bilinear form~$(\cdot, \cdot)_{C_h}$ changes to
\[
  (\varphi, \psi)_{C_h} = \frac{\eta}{h} \sum_{\kappa \in \mathcal{M}_I}  ( \jump{\partial_{\f n_{\kappa}} \varphi}_{\kappa}, \jump{\partial_{\f n_{\kappa}} \psi}_{\kappa} )_{L^2(I_{\kappa})}.
\]
In our study we vary the stability parameter over a range from $4h_0\cdot10^{-3}$ to $4h_0\cdot10^4$ and compute the errors for a fixed mesh size $h_0$. Figure 3 shows the errors in the $L^2$-, $H^1$- and $H^2$-norm for mesh size $h_0 = 1/2^6$ and different polynomial degrees. In comparison, we plot the error obtained with the approximate $C^1$ method with a dashed line, which is independent of the stability parameter. From the numerical results it is evident that Nitsche's method does not converge properly for large values of the stability parameter. The reason for this is that a too large parameter leads to an over-penalization of the jump of the normal derivative, which, in general, leads to locking of the solution. On the other hand, a too small stability parameter leads to instability. Also, we see a 'spike' occuring for a value of $\eta$ close to the constant $c_\kappa(h)$ in Lemma~\ref{lem:averageboundness}. A similar behavior with the occurence of 'spikes' is also observed in~\cite{embar2010imposing} where the authors use Nitsche's method for imposing Dirichlet boundary conditions.

\begin{figure}[h!]
%%
%% 1 row
%%
\begin{subfigure}[b]{0.23\textwidth}
  \centering
  
   \resizebox{\textwidth}{!}{
\includegraphics[width=\textwidth]{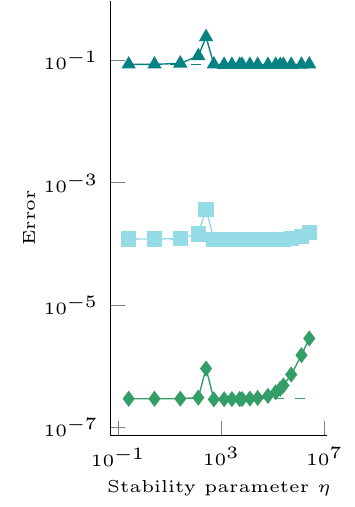}
    }
      
\subcaption{Ex.\ I, $p = 3, r = 2$} 
\end{subfigure}
\begin{subfigure}[b]{0.23\textwidth}
  \centering
  
   \resizebox{\textwidth}{!}{
\includegraphics[width=\textwidth]{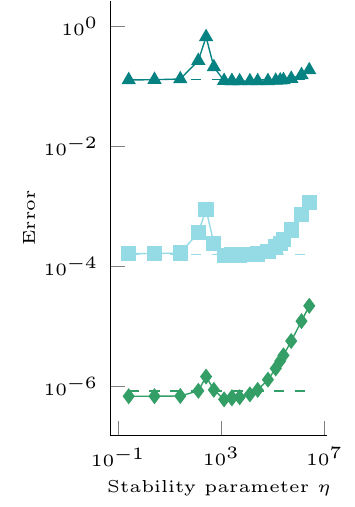}
    }
    
\subcaption{Ex.\ II, $p = 3, r = 2$}
\end{subfigure}
\begin{subfigure}[b]{0.23\textwidth}
  \centering
  
   \resizebox{\textwidth}{!}{
\includegraphics[width=\textwidth]{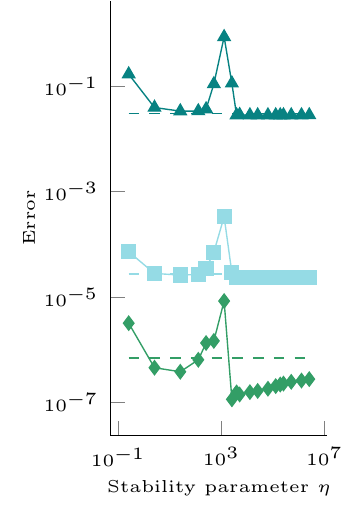}
    }
    
\subcaption{Ex.\ III, $p = 3, r = 2$}
\end{subfigure}
\begin{subfigure}[b]{0.23\textwidth}
  \centering
  
   \resizebox{\textwidth}{!}{
\includegraphics[width=\textwidth]{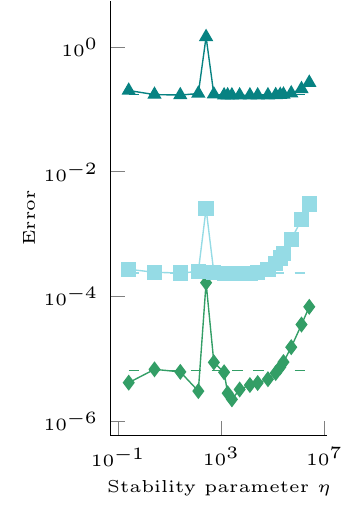}
    }
    
\subcaption{Ex.\ IV, $p = 3, r = 2$}
\end{subfigure}

%%
%% 2 row
%%
\begin{subfigure}[b]{0.23\textwidth}
  \centering
  
   \resizebox{\textwidth}{!}{
\includegraphics[width=\textwidth]{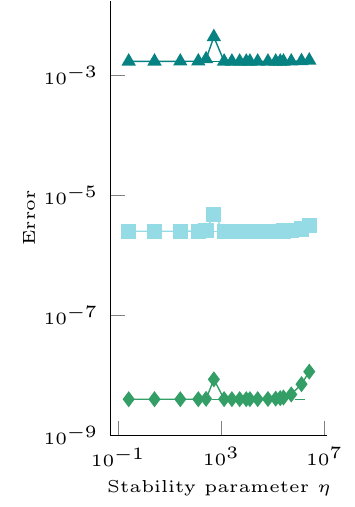}
    }
      
\subcaption{Ex.\ I, $p = 4, r = 3$} 
\end{subfigure}
\begin{subfigure}[b]{0.23\textwidth}
  \centering
  
   \resizebox{\textwidth}{!}{
\includegraphics[width=\textwidth]{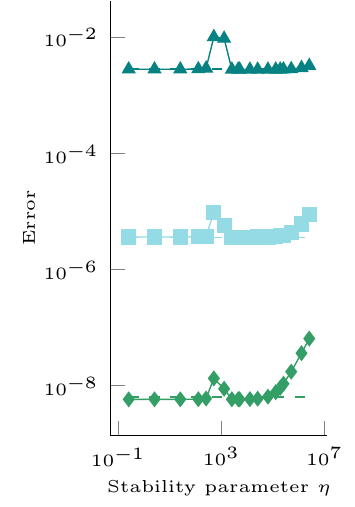}
    }
    
\subcaption{Ex.\ II, $p = 4, r = 3$}
\end{subfigure}
\begin{subfigure}[b]{0.23\textwidth}
  \centering
  
   \resizebox{\textwidth}{!}{
\includegraphics[width=\textwidth]{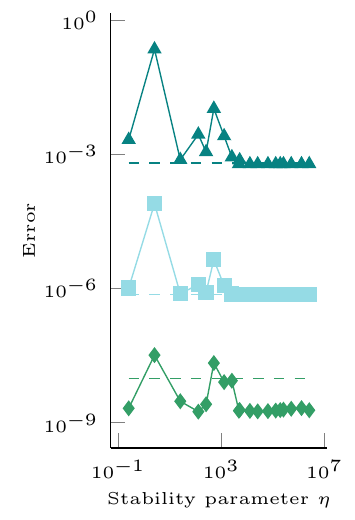}
    }
    
\subcaption{Ex.\ III, $p = 4, r = 3$}
\end{subfigure}
\begin{subfigure}[b]{0.23\textwidth}
  \centering
  
   \resizebox{\textwidth}{!}{
\includegraphics[width=\textwidth]{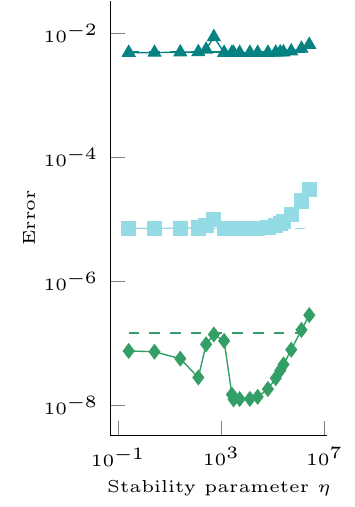}
    }
    
\subcaption{Ex.\ IV, $p = 4, r = 3$}
\end{subfigure}

%%
%% 3 row
%%
\begin{subfigure}[b]{0.23\textwidth}
  \centering
  
   \resizebox{\textwidth}{!}{
\includegraphics[width=\textwidth]{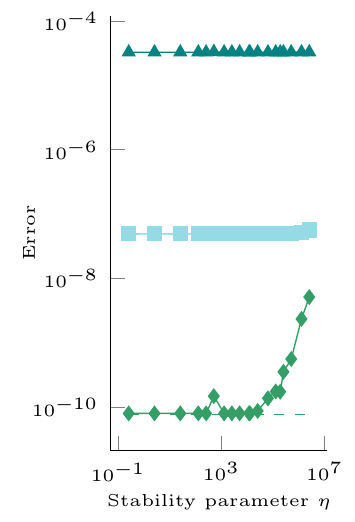}
    }
      
\subcaption{Ex.\ I, $p = 5, r = 4$} 
\end{subfigure}
\begin{subfigure}[b]{0.23\textwidth}
  \centering
  
   \resizebox{\textwidth}{!}{
\includegraphics[width=\textwidth]{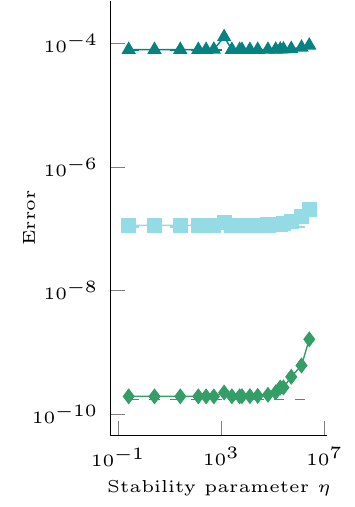}
    }
    
\subcaption{Ex.\ II, $p = 5, r = 4$}
\end{subfigure}
\begin{subfigure}[b]{0.23\textwidth}
  \centering
  
   \resizebox{\textwidth}{!}{
\includegraphics[width=\textwidth]{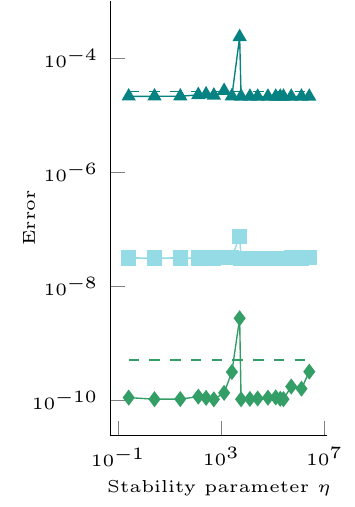}
    }
    
\subcaption{Ex.\ III, $p = 5, r = 4$}
\end{subfigure}
\begin{subfigure}[b]{0.23\textwidth}
  \centering
  
   \resizebox{\textwidth}{!}{
\includegraphics[width=\textwidth]{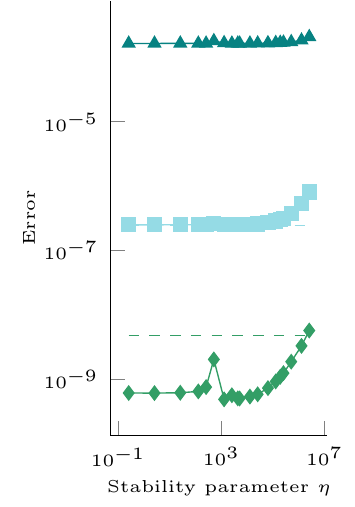}
    }
    
\subcaption{Ex.\ IV, $p = 5, r = 4$}
\end{subfigure}

\begin{center}
  \begin{subfigure}[t]{0.5\textwidth}
  \centering
  
   \resizebox{\textwidth}{!}{
\includegraphics[width=\textwidth]{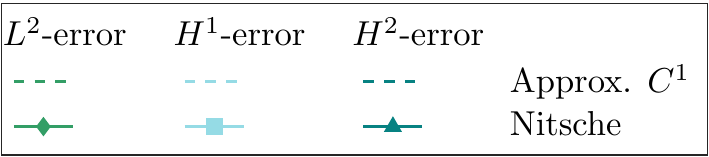}
    }

\end{subfigure}
\end{center}

\caption{The error for different stability parameters at fixed mesh size $h = 1/2^6$.}
\end{figure}  

\FloatBarrier

\subsection{Comparing an AS-$G^1$ reparametrization with the approximate $C^1$ discretization} \label{sec:reparam-v-approx}

In the last example, we follow the reparametrization strategy as in~\cite{kapl2018construction}. Here, a non AS-$G^1$ geometry is reparametrized into an AS-$G^1$ geometry. The consequence is that although the geometry and the discretization spaces can be constructed with exactly $C^1$-smooth basis functions as in~\cite{kapl2018construction,kapl2019argyris}, the geometry may change as a result of the reparametrization. We choose the geometry of Example~IV and reparametrized the non-AS-$G^1$ geometry. Figure~\ref{fig:asVsNonas} shows the results. There, the differences between the non-AS-$G^1$ geometry and the AS-$G^1$ geometry are shown represented in black (original) and red lines (reparametrized). For both methods we chose the maximum possible regularity. That is, we select $r = p - 1$ for the approximate~$C^1$ method and $r = p - 2$ for the AS-$G^1$ discretization -- which is a necessary restriction derived from the construction, cf.~\cite{collin2016analysis}. If $r = p - 1$ then the $C^1$-smooth subspace $\mathcal{X}_h \cap C^1(\Omega)$ reduces in general to global polynomials when restricted to one interface, which results in locking of the numerical solution. In Table~\ref{tab:asVsNonas}, for fixed mesh size $h=1/2^5$, the number of dofs and the errors are given. There we see that both methods yield very similar errors. However, the major difference between both methods is the number of dofs, which can be explained by the different regularities: since the AS-$G^1$ discretization requires $r=p-2$, the approximate $C^1$ method with $r=p-1$ needs much fewer dofs to obtain the same error levels.

\begin{figure}[hp!]
  \centering
  \definecolor{red}{HTML}{FF0000}
   \resizebox{0.5\textwidth}{!}{
\includegraphics[width=\textwidth]{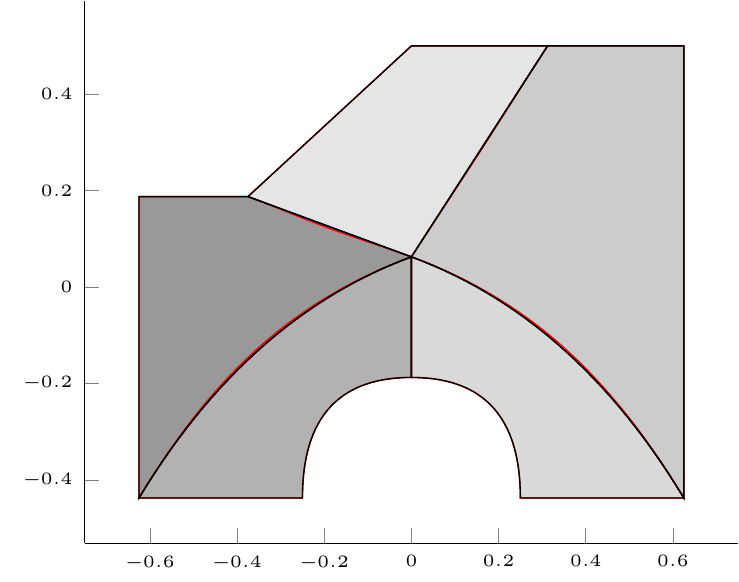}
    }
      
\caption{Non-AS-$G^1$ geometry (black) vs. AS-$G^1$ reparametrization (red)} \label{fig:asVsNonas}
\end{figure}

\begin{table}[h!]
\centering
  \begin{tabular}{c|c|c|c|c|c}
        & $h$ & \# dofs & $L^2$-error & $H^1$-error & $H^2$-error \\ \hline
   AS-$G^1$ geometry $p = 3$, $r = 1$ & 0.03125 & 19835 & 8.24069e-05 & 2.21436e-03 & 7.94010e-01 \\
   non-AS-$G^1$ geometry $p = 3$, $r = 2$ & 0.03125 & 5420 & 6.67878e-05 & 2.05893e-03 & 7.34844e-01 \\ \hline
      AS-$G^1$ geometry $p = 4$, $r = 2$ & 0.03125 & 21135 & 3.37653e-06 & 1.17475e-04 & 4.60163e-02 \\
   non AS-$G^1$ geometry $p = 4$, $r = 3$ & 0.03125 & 5755 & 3.19555e-06 & 1.20678e-04 & 4.11154e-02 \\ \hline
      AS-$G^1$ geometry $p = 5$, $r = 3$ & 0.03125 & 22475 & 1.65998e-07 & 7.97048e-06 & 3.19901e-03 \\
   non AS-$G^1$ geometry $p = 5$, $r = 4$ & 0.03125 & 6100 & 1.69152e-07 & 8.25093e-06 & 2.76540e-03 \\ 
  \end{tabular}
  \caption{The dofs and errors for the AS-$G^1$ geometry and the non-AS-$G^1$ geometry.} \label{tab:asVsNonas}
\end{table}

\section{Conclusion and future work}  \label{sec:conclusion}

We extend the basic construction from~\cite{WEINMULLER2021114017} to general multi-patch domains. Therefore, we introduce a construction for basis functions  around vertices using interpolation of functions that are approximately $C^1$-smooth across interfaces. Three different kinds of spaces are created in the construction: the patch interior spaces, the edge (interface and boundary) spaces and the corner spaces (both for boundary vertices and inner vertices), which are derived from the topology of the multi-patch geometry. This creates spaces that locally possess higher polynomial degrees and lower regularity, with the exception of the patch interior space, which is a standard isogeometric space. As a result we get non-nested spaces. In contrast to discretization spaces over AS-$G^1$ parametrizations, as in~\cite{kapl2019argyris}, which require $r \leq p - 2$, the approximate $C^1$ method also allows us to choose spline spaces of maximum regularity $r=p-1$.

Moreover, we compare the approximate $C^1$ method with Nitsche's method. In the numerical experiments we see that both methods converge optimally and the error values are almost the same. While one has to determined a suitable stability parameter for Nitsche's method, no such tuning is needed for the approximate $C^1$ method. Thus, to summarize, the approximate $C^1$ method provides an explicit and simple to implement alternative to weak ($H^2$-nonconforming) and exact ($H^2$-conforming) methods to solve fourth order problems, exemplified on a biharmonic model problem. The advantages of the approximate $C^1$ method are that the method can be applied on any $C^0$-conforming multi-patch parametrization and does not depend on any non-trivial parameter choices.

In the future we want to study several aspects of the method, such as convergence and stability properties and extensions that result in nested spaces. This would allow an adaptive construction with THB-splines, following the work as in~\cite{bracco2020isogeometric,bracco2021c1}. Moreover, we want to extend the construction to $C^0$-non-conforming (non-matching) interfaces and to surface domains. In such a context, the approximate $C^1$ method could be a viable option to discretize Kirchhoff--Love shell problems. Another possible direction of research is the extension to volumetric domains, where $C^1$-smooth discretizations, in general, yield suboptimal convergence rates, cf.~\cite{kapl2021c}. Since the approximate $C^1$ method has no interface integrals, it would be interesting to combine the method with a multigrid solver, see~\cite{sogn2019robust,sogn2021multigrid}.

\section*{Acknowledgments}

Both authors are supported by the Austrian Science Fund (FWF) and the government of Upper Austria through the project P~30926-NBL entitled ``Weak and approximate $C^1$ smoothness in isogeometric analysis''. Moreover, Thomas Takacs is partially supported by the Linz Institute of Technology (LIT) and the government of Upper Austria through the project LIT-2019-8-SEE-116 entitled ``PARTITION – PDE-aware isogeometric discretization based on neural networks''. All support is gratefully acknowledged.

\appendix

\section{The $C^2$ interpolation at the vertex}\label{sec:appendix-C2-interpolation}

In the following, we summarize the $C^2$ interpolation from~\cite{kapl2019argyris,hna2021}. For simplicity of the notation, we assume that the vertex is at $\boldsymbol{x} = V^{(k)}_1 = \f F^{(k)} (0,0)$. We repeat again the three sets of basis functions
\[
 \widehat{\mathcal{B}}_{b.e.} = \{ f_1^{(k)} [b^+_1,0],f_1^{(k)} [b^+_2,0],f_1^{(k)} [b^+_3,0],f_1^{(k)} [0,b^-_1],f_1^{(k)} [0,b^-_2], {\f b}^{(k)}_{(1,3)} \},
\]
corresponding to the bottom edge,
\[
 \widehat{\mathcal{B}}_{l.e.} = \{ f_4^{(k)} [b^+_1,0],f_4^{(k)} [b^+_2,0],f_4^{(k)} [b^+_3,0],f_4^{(k)} [0,b^-_1],f_4^{(k)} [0,b^-_2], {\f b}^{(k)}_{(3,1)} \},
\]
corresponding to the left edge, as well as 
\[
 \widehat{\mathcal{B}}_{c.t.} = \{ {\f b}^{(k)}_{(1,1)},{\f b}^{(k)}_{(1,2)},{\f b}^{(k)}_{(1,3)},{\f b}^{(k)}_{(2,1)},{\f b}^{(k)}_{(2,2)}, {\f b}^{(k)}_{(3,1)} \}.
\]
Then we predefine the projection operator $\Pi_1^{(k)} : C^2 (\boldsymbol{x}) \to {\mathcal{A}}_{V,s}^{(k)}$ such that for all $\varphi \in C^2 (\boldsymbol{x})$ it holds
\begin{align*}
  \Pi_1^{(k)} \varphi (\boldsymbol{x}) &= \varphi  (\boldsymbol{x}), \\
  \nabla_{\boldsymbol{x}} \Pi_1^{(k)} \varphi (\boldsymbol{x}) &= \nabla_{\boldsymbol{x}} \varphi (\boldsymbol{x}), \\
  \text{Hess} ( \Pi_1^{(k)} \varphi (\boldsymbol{x}) ) &= \text{Hess} ( \varphi (\boldsymbol{x}) ).
\end{align*}
We get the projection operator $\Pi_1^{(k)}$ by introducing the $C^2$ interpolation in the physical domain for the three spaces $\mathcal{B}_{b.e.}$, $\mathcal{B}_{l.e.}$ and $\mathcal{B}_{c.t.}$ denoted by $\Pi_{b.e.}$, $\Pi_{l.e.}$ and $\Pi_{c.t.}$, respectively. Then we add the first two interpolations and subtract the third to obtain the six basis functions for the vertex space ${\mathcal{A}}_{V,s}^{(k)}$.

We introduce the $C^2$ interpolation for the bottom edge and obtain the unique projector $\Pi_{b.e.} : C^2 (\boldsymbol{x}) \to \mbox{span}\{\mathcal{B}_{b.e.}\}$ which satisfies
\begin{align*}
  \Pi_{b.e.} \varphi (\boldsymbol{x}) &= \varphi  (\boldsymbol{x}), \\
  \nabla_{\boldsymbol{x}} \Pi_{b.e.} \varphi (\boldsymbol{x}) &= \nabla_{\boldsymbol{x}} \varphi (\boldsymbol{x}), \\
  \text{Hess} ( \Pi_{b.e.} \varphi (\boldsymbol{x}) ) &= \text{Hess} ( \varphi (\boldsymbol{x}) ).
\end{align*}
The projection operators $\Pi_{l.e.} : C^2 (\boldsymbol{x}) \to \mbox{span}\{\mathcal{B}_{l.e.}\}$ and $\Pi_{c.t.} : C^2 (\boldsymbol{x}) \to \mbox{span}\{\mathcal{B}_{c.t.}\}$ are defined analogously.
Then the operator $\Pi_1^{(k)}$ is defined as
\begin{align*}
 \Pi_1^{(k)} = \Pi_{b.e.} + \Pi_{l.e.} - \Pi_{c.t.}
\end{align*}
which concludes the $C^2$ interpolation. Thus, the space $\widehat{\mathcal{A}}_{V,s}^{(k)}$ is defined implicitly through the interpolation. In~\cite{kapl2019argyris,hna2021}, an explicit formula of the vertex basis functions is stated.

\end{document}